\documentclass[10pt, a4paper, reqno, oneside]{amsart}

\usepackage{amsmath, amsfonts, amsthm, amssymb, mathtools, enumerate, multicol, scalefnt, relsize}
\usepackage[mathscr]{euscript}
\usepackage{mathbbol}
\usepackage[latin1]{inputenc}
\usepackage{graphicx}
\usepackage[all]{xy}
\usepackage{enumitem}
\usepackage{autobreak,lipsum}
\setlist[itemize]{noitemsep, topsep=1pt, leftmargin=20pt}
\usepackage{xfrac}
\usepackage{todonotes}

\usepackage{fullpage}
\usepackage{hyperref}

\newcommand\bcdot{\ensuremath{
		\mathchoice
		{\mskip\thinmuskip\lower0.2ex\hbox{\scalebox{1.6}{$\cdot$}}\mskip\thinmuskip}}
	{\mskip\thinmuskip\lower0.2ex\hbox{\scalebox{1.6}{$\cdot$}}\mskip\thinmuskip}
	{\lower0.3ex\hbox{\scalebox{1.2}{$\cdot$}}}
	{\lower0.3ex\hbox{\scalebox{1.2}{$\cdot$}}}
}

\theoremstyle{plain}
\newtheorem{theo}{Theorem}[section]

\newtheorem{prop}[theo]{Proposition}

\theoremstyle{definition}
\newtheorem{rem}[theo]{Remark}

\newtheorem{definition}[theo]{Definition}

\theoremstyle{plain}
\newtheorem{lemma}[theo]{Lemma}
\newtheorem{theorem}[theo]{Theorem}

\newtheorem{proposition}[theo]{Proposition}

\theoremstyle{definition}

\newtheorem{remark}[theo]{Remark}

\theoremstyle{plain}
\newtheorem{thmint}{Theorem}

\theoremstyle{definition}
\newtheorem*{definition*}{Definition}

\DeclareSymbolFontAlphabet{\mathbb}{AMSb}
\DeclareSymbolFontAlphabet{\mathbbl}{bbold}
\DeclareMathOperator\Lie{Lie}
\DeclareMathOperator\sgn{sgn}

\makeatletter
\@namedef{subjclassname@2020}{\textup{2020} Mathematics Subject Classification}
\makeatother

\allowdisplaybreaks

\title[]{Pluriclosed manifolds with parallel Bismut torsion}

\author{Giuseppe Barbaro}
\address[Giuseppe Barbaro]{Department of Mathematics, Aarhus University, Ny Munkegade 118, 8000 Aarhus C, Denmark}
\email{g.barbaro@math.au.dk}

\author{Francesco Pediconi}
\address[Francesco Pediconi]{Dipartimento di Scienze Matematiche ``Giuseppe Luigi Lagrange'' \\ Politecnico di Torino, corso Duca degli Abruzzi 24, 10129 Torino, Italy}
\email{francesco.pediconi@polito.it}

\author{Nicoletta Tardini}
\address[Nicoletta Tardini]{Dipartimento di Scienze Matematiche, Fisiche e Informatiche \\ Universit\`a degli Studi di Parma, Parco Area delle Scienze 53/a, 43124 Parma, Italy}
\email{nicoletta.tardini@unipr.it}

\thanks{All authors are members of GNSAGA of INdAM. The first-named author has been supported by Villum Young Investigator 0019098. The second-named and third-named authors have been supported by the PRIN 2022 project ``Real and Complex Manifolds: Geometry and Holomorphic Dynamics'' (code 2022AP8HZ9). The third-named author was supported by University of Parma through the action Bando di Ateneo 2023 per la ricerca.}

\begin{document}

\begin{abstract}
We present a complete classification of simply-connected pluriclosed manifolds with parallel Bismut torsion, extending previously known results in the literature. Consequently, we also establish a splitting theorem for compact manifolds that are both pluriclosed with parallel Bismut torsion and Calabi--Yau with torsion.
\end{abstract}

\maketitle

\section{Introduction}
\setcounter{equation} 0

Given a complex manifold $(M,J)$, there always exists a Hermitian metric on it, {\it i.e.}, a Riemannian metric $g$ such that $g(J\cdot , J\cdot ) = g(\cdot , \cdot )$. The metric $g$ is K\"ahler if the associated fundamental $2$-form $\omega(\cdot,\cdot)=g(J\cdot,\cdot)$ is closed and, consequently, K\"ahler geometry lies in the intersection of complex, Riemannian, and symplectic geometry. Therefore, the existence of a K\"ahler metric imposes strong restrictions on the topology and on the geometry of $(M,J)$ and so, from a complex-geometric point of view, it is relevant to study Hermitian non-K\"ahler manifolds. \smallskip

In the K\"ahler case the Levi-Civita connection $D$ preserves the complex structure, {\it i.e.}, $DJ=0$, and so it encodes both the Riemannian and the complex geometry of the manifold. Therefore, to study non-K\"ahler manifolds, other natural connections with non-vanishing torsion that preserve both the Riemannian metric $g$ and the complex structure $J$ have been introduced. These are called {\it Hermitian}, and, among them, a distinguished $1$-parameter family of {\it canonical} connections has been presented by Gauduchon \cite{MR1456265} by requiring a condition on the torsion. Among these, the {\it Bismut connection} $\nabla$ (also known as {\it Strominger connection}, or {\it K\"ahler with torsion connection}) is the unique Hermitian connection whose torsion $T$ is {\it totally skew-symmetric}, namely, $g(T(\cdot,\cdot),\cdot)$ is a $3$-form. Indeed, $T$ verifies
$$
g(T(\,\cdot\,,\,\cdot\,),\,\cdot\,) = {\rm d}\omega(J\,\cdot\,, J\,\cdot\,, J\,\cdot\,) = J^{-1}{\rm d}\omega(\,\cdot\,, \,\cdot\,, \,\cdot\,) \,\, .
$$
The Bismut connection is important due to its wide range of applications. For instance, it appears in non-K\"ahler index theory \cite{MR1006380, MR3099098, MR4696610} and in the context of geometrization of complex surfaces through the pluriclosed flow \cite{MR4181011, MR2673720, MR3110582}. Additionally, it plays a role in string theory \cite{MR800347, GATES1984157, MR872720, MR851702}. \smallskip

We remark that metric connections with totally skew-symmetric torsion have appeared implicitly in the last century before the introduction of the Bismut connection (see {\it e.g.\ }\cite{MR107275,36,37}). Moreover, they have been extensively used as a tool in the study of manifolds with non-integrable $\mathsf{G}$-structures (see {\it e.g.\ }\cite{MR1928632, MR1958088, MR2047649, MR2038309, MR2067465, MR2114426, MR2265474, MR2322400, ivanov2023riemannian}). Among metric connections with totally skew-symmetric torsion, two distinguished classes play an important role: those with {\it closed torsion} and those with {\it parallel torsion}. In the context of complex geometry, the Hermitian metrics whose Bismut connection has closed torsion, namely ${\rm d}J^{-1}{\rm d}\omega=0$, are called {\it pluriclosed} \cite{MR1006380} (also known as {\it strong K\"ahler with torsion}). In this paper, we give a complete classification of simply-connected pluriclosed manifolds with parallel Bismut torsion.

\begin{thmint} \label{thm:MAIN-A}
Let $(M^{2n}, J, g)$ be a complete, simply-connected Hermitian manifold. Then, $(M, J, g)$ is pluriclosed with parallel Bismut torsion if and only if it decomposes as a product of Hermitian irreducible factors, each of them is either K\"ahler or a Riemannian product
\begin{equation} \label{eq:usualproduct} \tag{$\star$}
\mathbb{R}^{\ell} \times \prod_{i=1}^{s} S_{i} \times \mathsf{K}
\end{equation}
endowed with a standard complex structure, where $\mathbb{R}^{\ell}$ is the $\ell$-dimensional flat Euclidean space, each $S_{i}$ is a Sasaki $3$-dimensional manifold, $\mathsf{K}$ is a compact semisimple Lie group of rank $r$ with a bi-invariant metric and $\ell \leq s+r$.
\end{thmint}

Here, by {\it standard complex structure} on \eqref{eq:usualproduct}, we mean an integrable almost complex structure $J$ that is compatible with the transverse complex distribution on each Sasaki factor, projects onto a $\mathsf{K}$-invariant complex structure on the full flag manifold $\mathsf{K}/\mathsf{T}$ for a chosen maximal torus $\mathsf{T} \subset \mathsf{K}$ and preserves the distribution spanned by the Euclidean factor, the Reeb vector fields and the torus $\mathsf{T}$ (see Definition \ref{def:standard}).

The fact that Riemannian manifolds \eqref{eq:usualproduct} endowed with a standard complex structure are pluriclosed with parallel Bismut torsion follows by direct computation. On the other hand, the converse implication in Theorem \ref{thm:MAIN-A} relies on two crucial algebraic facts. The first one is that the two conditions ${\rm d}T = 0$ and $\nabla T = 0$ imply that the torsion $T$ satisfies the Jacobi identity, hence it gives rise to a Lie algebra bundle on $TM$ with typical fiber $\mathfrak{g}$. Moreover, the Hermitian structure $(J,g)$ induces both a bi-invariant metric and an orthogonal linear complex structure on $\mathfrak{g}$. Therefore, the root spaces decomposition of $\mathfrak{g}$ determines a de Rham decomposition of the manifold $(M,J,g)$. The second one is that each de Rham factor corresponding to an ideal of $\mathfrak{g}$ with rank at least $2$ is necessarily Bismut flat, while the others are $3$-dimensional Sasaki manifolds. Finally, by \cite[Theorem 2.2]{MR2651537}, the Bismut flat factor is a Lie group with a bi-invariant metric.

Notice that results for irreducible Riemannian manifolds endowed with metric connections with closed and parallel totally skew-symmetric torsion appeared in \cite{MR3319112}. However, in order to obtain Theorem \ref{thm:MAIN-A}, it is crucial to exploit the datum of the parallel complex structure that induces further rigidity for pluriclosed manifolds with parallel Bismut torsion. \smallskip

Notice that pluriclosed manifolds with parallel Bismut torsion were classified, up to complex dimension five, in \cite{MR4554474, MR4577328, ZZ4-5} (see also Section \ref{sect:low-dim}). The case of surfaces is particularly relevant since, in that case, these coupled conditions are equivalent to being {\it Vaisman} \cite[Theorem 2]{MR4554474}. Furthermore, Vaisman metrics always have parallel Bismut torsion \cite[Corollary 3.8]{MR4480223} but, in complex dimension greater than $2$, Vaisman manifolds do not admit pluriclosed metrics \cite[Theorem 3.1]{angella-otiman} (see also \cite[Theorem 4]{MR4577328}). Notice also that the argument in \cite{angella-otiman} extends to any smooth modification of a Vaisman manifold. This gives partial evidence to \cite[Conjecture 1]{MR4577328}, where the authors suggest that compact pluriclosed manifolds with parallel Bismut torsion of dimension greater than $2$ do not admit {\it locally conformally K\"ahler metrics}. We believe that Theorem \ref{thm:MAIN-A} could provide more insights on this conjecture. \smallskip

We remark that the conditions ${\rm d}T = 0$ and $\nabla T = 0$ coupled together have appeared in the study of symmetries of the curvature tensor of the Bismut connection. Let us recall that, following \cite[Definition 4 and Remark 3]{MR4564028}, a Hermitian manifold $(M, J, g)$ is called {\it Bismut K\"ahler-like} ({\it BKL} for short) if the curvature tensor $R$ of its Bismut connection satisfies the {\it torsionless first Bianchi identity}
\begin{equation} \label{eq:1bncintro} \tag{1-Bnc}
R_{X,Y}Z +R_{Y,Z}X +R_{Z,X}Y = 0 \,\, .
\end{equation}
Then, as a result of Zhao and Zheng \cite[Theorem 1]{MR4554474}, a Hermitian manifold $(M,J,g)$ is pluriclosed with parallel Bismut torsion if and only if it is BKL. \smallskip

The study of Hermitian manifolds whose curvature tensor, with respect to the Levi-Civita or the Chern connections, satisfies further symmetries, was initiated by Gray \cite{MR436054}, and then it was recently studied also in \cite{MR3900484,MR3632564}. This study was then extended in \cite{MR4564028} to a class of metric connections that includes Gauduchon's canonical ones. Accordingly, a Gauduchon connection is called {\it K\"ahler-like} if its curvature tensor satisfies \eqref{eq:1bncintro}. As a matter of fact, on a K\"ahler manifold, all the canonical connections have vanishing torsion, and so \eqref{eq:1bncintro} coincides with the first Bianchi identity. Clearly, \eqref{eq:1bncintro} is trivially satisfied in the case of flat curvature, that has been studied {\it e.g.\ }in \cite{MR103990, MR4127891, MR2247437, MR2795448, MR3843433, MR4002291, MR4058531}. Recently, it has been proved in \cite[Theorem 3.1]{lafuente2022hermitian} that for Gauduchon connections which are neither the Chern nor the Bismut one, the K\"ahler-like condition on compact manifolds forces the metric to be K\"ahler. Therefore, the only possible non-K\"ahler compact manifolds that admit this kind of curvature symmetry are BKL or {\it Chern K\"ahler-like}. We remark that Chern K\"ahler-like manifolds are known to be {\it balanced} \cite[Theorem 3]{MR3900484}, but, to the best of our knowledge, a non-Chern flat example is still missing.

Simply-connected Bismut flat manifolds were classified in \cite{MR4127891}, where the authors showed that they are holomorphically isometric to the so called {\it Samelson spaces}. We recall that a Samelson space is a simply-connected even dimensional Lie group $\mathbb{R}^{\ell} \times \mathsf{K}$ endowed with a bi-invariant metric and an orthogonal left-invariant complex structure, where $\mathsf{K}$ is compact and semisimple. These manifolds are known to be pluriclosed (see {\it e.g.\ }\cite{MR1836272}), and we point out that there are a few examples of pluriclosed metrics that are not locally homogeneous. Concerning this, Theorem \ref{thm:MAIN-A} provides new examples of pluriclosed manifolds that are not locally homogeneous.

Bismut flat manifolds are related to Yau's Problem 87 \cite{MR1216573} concerning compact Hermitian manifolds with holonomy reduced to subgroups of $\mathsf{U}(n)$. Indeed, the holonomy of the Bismut connection is always contained in $\mathsf{U}(n)$, and it reduces to $\mathsf{SU}(n)$ if and only if its {\it Bismut Ricci form} $\rho$, obtained by tracing the Bismut curvature tensor $R$ in the endomorphism components, vanishes.
These manifolds are called {\it Calabi--Yau with torsion} ({\it CYT} for short). When the CYT condition is coupled with the BKL condition on a simply-connected manifold, then one obtains a product of a K\"ahler Ricci flat factor and a Samelson space (see \cite[Theorem 3]{ZZ4-5} and \cite[Theorem 1.1]{brienza2024cyt}). As a direct consequence of Theorem \ref{thm:MAIN-A}, we recover this result for simply-connected manifolds. Furthermore, by studying their fundamental groups, we also characterize compact BKL and CYT manifolds, up to finite covers.

\begin{thmint} \label{thm:MAIN-B}
Let $(M^{2n}, J, g)$ be a compact Hermitian manifold. Then, $(M, J, g)$ is CYT and pluriclosed with parallel Bismut torsion if and only if it splits, up to a finite cover, as a Hermitian product
$$
\mathcal{Y} \times M' \,\, ,
$$
where $\mathcal{Y}$ is K\"ahler Ricci flat and $M'$ is a local Samelson space.
\end{thmint}

Here, following \cite{MR4127891}, a {\it local Samelson space} is the quotient $M' = \big(\mathbb{R}^{\ell} \times \mathsf{K}\big) / \mathbb{Z}^{\ell}$ of a Samelson space by a free abelian group of rank $\ell$ acting as $\gamma \cdot (t,a) \coloneqq (t +\gamma, \psi(\gamma).a)$ for some group homomorphism $\psi: \mathbb{Z}^{\ell} \to {\rm Iso}(\mathsf{K})$ into the isometry group ${\rm Iso}(\mathsf{K})$ of $\mathsf{K}$. \smallskip

CYT manifolds appeared in the Physics literature after the works of Strominger \cite{MR851702} and Hull \cite{MR862401}, and have been extensively studied (see {\it e.g.\ }\cite{MR2015241, MR2101226, MR2406264, MR2521810, MR2764884, MR3643933, MR3972000, MR4592898, MR4554058, MR4614394, barbaro2023bismut, brienza2024cyt}). Moreover, pluriclosed CYT metrics (known also as {\it Bismut Hermitian Einstein} \cite[Definition 8.11]{MR4284898} are the only non-K\"ahler static points for the {\it pluriclosed flow}, that is a generalization of the classical K\"ahler--Ricci flow introduced by Streets and Tian \cite{MR2673720}, and this motivates the search for explicit examples that are neither K\"ahler, nor Bismut flat, nor a product of them. The first example, which is not topologically a product, also has parallel Bismut torsion and has been found in \cite{brienza2024cyt}. Notice that, by means of Theorem \ref{thm:MAIN-B}, compact Bismut Hermitian Einstein with parallel Bismut torsion are always finitely covered by Hermitian products of K\"ahler Ricci flat manifolds and Bismut flat manifolds, and this further motivates the search of Bismut Hermitian Einstein examples without parallel Bismut torsion.

Finally, in \cite{MR4257077, MR4517715, MR4732911}, the authors have shown evidence that the pluriclosed flow could preserve the BKL condition. We believe that Theorem \ref{thm:MAIN-A} could provide insights in this direction, and it will be the object of further studies. \medskip

The paper is organized as follows. Section \ref{sect:prel} contains background material on the Bismut connection and on compact Lie algebras. In Section \ref{sect:if-part}, we provide a construction for a general class of pluriclosed manifolds with parallel Bismut torsion. Section \ref{sect:classification} contains the proof of Theorem \ref{thm:MAIN-A}. In Section \ref{sect:low-dim}, we explain how Theorem \ref{thm:MAIN-A} extends all the previously known classification results in low dimension. Finally, in Section \ref{sect:BKL+CYT}, we prove Theorem \ref{thm:MAIN-B}. \medskip

\noindent {\itshape Acknowledgements.\ } The authors would like to thank Ilka Agricola, Daniele Angella, Leander Stecker, Luigi Verdiani and Fangyang Zheng for their comments and suggestions. They also thank the anonymous referee for their careful reading of the manuscript and for their constructive comments.

\medskip
\section{Preliminaries}
\label{sect:prel} \setcounter{equation} 0

In this section, we review some preliminaries on the Bismut connection and Lie theory that will be useful in the sequel.

\subsection{The Bismut connection} \label{sect:prelBKL} \hfill \par

Let $(M^{2n},J,g)$ be a Hermitian manifold. For the sake of notation, we denote by $\Gamma(\EuScript{U}; E)$ the space of local smooth sections of any vector bundle $E \to M$ over $M$ defined on an open set $\EuScript{U} \subset M$. If $\mathcal{D} \subset TM$ is a smooth distribution of the tangent bundle, we denote by
$$
\Gamma(\EuScript{U};\mathcal{D}) \coloneqq \{X \in \Gamma(\EuScript{U};TM) : \text{$X|_p \in \mathcal{D}_p$ for any $p \in \EuScript{U}$}\}
$$
the space of sections of $\mathcal{D}$ defined on $\EuScript{U}$. Finally, if $X$ is a vector field on $M$ and $\Phi$ is a tensor field on $M$ of type $(p,q)$, with $q \geq 1$, we denote by $X\,\lrcorner\,\Phi$ the tensor field on $M$ of type $(p,q-1)$ defined by 
$$
(X\,\lrcorner\,\Phi)(V_1,{\dots},V_{q-1}) \coloneqq \Phi(X,V_1,{\dots},V_{q-1}) \,\, .
$$

By the Newlander--Nirenberg Theorem, the fact that $J$ is integrable is equivalent to the vanishing of the Nijenhuis tensor $N_J$ of $J$, {\it i.e.},
$$
N_J(X,Y) \coloneqq [JX,JY] -[X,Y] -J[JX,Y] -J[X,JY] = 0 \,\, .
$$
We consider the action of $J$ on covectors $\vartheta \in \Gamma(\EuScript{U};T^*M)$ defined by $J\vartheta \coloneqq \vartheta(J^{-1}\,\cdot\,)$, which extends on any $k$-form on $M$. We denote by $\omega \coloneqq g(J\,\cdot\, , \,\cdot\, )$ the corresponding fundamental $2$-form and by ${\rm d}^c$ the real operator defined by ${\rm d}^c \coloneqq - J^{-1} \circ {\rm d} \circ J$, so that
\begin{equation} \label{eq:torsionBismut}
	{\rm d}^c\omega(X,Y,Z) = -{\rm d}\omega(JX,JY,JZ) \,\, .
\end{equation}
For later use, we prove the following technical result.

\begin{lemma}
For any $X, Y, Z \in \Gamma(\EuScript{U};TM)$, the following equation holds true:
\begin{equation} \label{eq:dwJ}
{\rm d}\omega(JX,JY,JZ) = {\rm d}\omega(X,Y,JZ) +{\rm d}\omega(X,JY,Z) +{\rm d}\omega(JX,Y,Z) \,\, .
\end{equation}
\end{lemma}

\begin{proof}
A direct computation shows that
$$\begin{aligned}
{\rm d}\omega(JX,JY,JZ) &= \mathcal{L}_{JX}\{g(JY,Z)\} +\mathcal{L}_{JY}\{g(JZ,X)\} +\mathcal{L}_{JZ}\{g(JX,Y)\} \\
&\qquad\qquad -g([JX,JY],Z) -g([JY,JZ],X) -g([JZ,JX],Y) \,\, , \\
{\rm d}\omega(JX,Y,Z) &= \mathcal{L}_{JX}\{g(JY,Z)\} +\mathcal{L}_Y\{g(Z,X)\} -\mathcal{L}_Z\{g(X,Y)\} \\
&\qquad\qquad +g([JX,Y],JZ) -g([Y,Z],X) +g([Z,JX],JY) \,\, , \\
{\rm d}\omega(X,JY,Z) &= -\mathcal{L}_X\{g(Y,Z)\} +\mathcal{L}_{JY}\{g(JZ,X)\} +\mathcal{L}_Z\{g(X,Y)\} \\
&\qquad\qquad +g([X,JY],JZ) +g([JY,Z],JX) -g([Z,X],Y) \,\, , \\
{\rm d}\omega(X,Y,JZ) &= \mathcal{L}_X\{g(Y,Z)\} -\mathcal{L}_Y\{g(Z,X)\} +\mathcal{L}_{JZ}\{g(JX,Y)\} \\
&\qquad\qquad -g([X,Y],Z) +g([Y,JZ],JX) +g([JZ,X],JY) \,\, .
\end{aligned}$$
Therefore, we get
$$\begin{aligned}
&{\rm d}\omega(JX,JY,JZ) -{\rm d}\omega(JX,Y,Z) -{\rm d}\omega(X,JY,Z) -{\rm d}\omega(X,Y,JZ) = \\
&\qquad\qquad\qquad\qquad\qquad\qquad = -g(N_J(X,Y),Z) -g(N_J(Y,Z),X) -g(N_J(Z,X),Y) \\
&\qquad\qquad\qquad\qquad\qquad\qquad = 0
\end{aligned}$$
and so \eqref{eq:dwJ} follows.
\end{proof}

We denote by $D$ the Levi-Civita connection of $(M,g)$ and by $\nabla$ the {\it Bismut connection} of $(M,J,g)$ defined by
\begin{equation} \label{eq:Bismutconn}
g(\nabla_XY,Z) \coloneqq g(D_XY,Z) -\tfrac12{\rm d}^c\omega(X,Y,Z) \,\, .
\end{equation}
We denote by $T$ the torsion tensor of $\nabla$, {\it i.e.},
$$
T(X,Y) \coloneqq \nabla_XY -\nabla_YX -[X,Y] \,\, ,
$$
and we remark that an easy computation shows that
$$
g(T(X,Y),Z) = -{\rm d}^c\omega(X,Y,Z) \,\ .
$$
For this reason, the Bismut connection is said to have {\it totally skew-symmetric torsion}. By \cite{MR1006380, MR1456265}, the Bismut connection is the unique linear connection with totally skew-symmetric torsion that preserves the Hermitian structure, {\it i.e.}, $\nabla J=0$ and $\nabla g=0$. We also denote by $R$ the curvature tensor of $\nabla$ defined as
$$
R_{X,Y}Z \coloneqq \nabla_X\nabla_YZ -\nabla_Y\nabla_XZ -\nabla_{[X,Y]}Z \,\, .
$$
For the sake of notation, we will also denote by $R$ the corresponding covariant tensor
$$
R(X,Y,Z,V) \coloneqq g(R_{X,Y}Z,V)
$$
and by $\rho$ the {\it (first) Bismut Ricci form} defined by
\begin{equation} \label{eq:BRicci}
\rho(X,Y) \coloneqq \sum_{i=1}^n R(X,Y,Je_i,e_i) \,\, ,
\end{equation}
where $(e_i,Je_i)_{i=1,{\dots},n}$ is a local unitary frame. \smallskip

The connection $\nabla$ is said to satisfy the {\it torsionless first Bianchi identity} if
\begin{equation} \label{eq:1bianchi} \tag{1-Bnc}
R_{X,Y}Z +R_{Y,Z}X +R_{Z,X}Y = 0
\end{equation}
for any $X,Y,Z \in \Gamma(\EuScript{U};TM)$. We stress the following

\begin{lemma} \label{lem:surprise}
Assume that \eqref{eq:1bianchi} is satisfied. Then, for any $X, Y, Z, V \in \Gamma(\EuScript{U};TM)$, the following equations are satisfied:
\begin{align}
R(X,Y,Z,V) &= R(Z,V,X,Y) \,\, , \label{eq:symR} \\
R_{JX,JY} &= R_{X,Y} \,\, . \label{eq:Rcomplex}
\end{align}
\end{lemma}

\begin{proof}
Since $\nabla$ preserves the metric $g$, \eqref{eq:symR} is a well-known algebraic consequence of \eqref{eq:1bianchi}. Moreover, since $\nabla$ preserves also $J$, \eqref{eq:Rcomplex} follows from \eqref{eq:symR}.
\end{proof}

We recall then the following

\begin{definition}[Definition 4 in \cite{MR4564028}]
The Hermitian manifold $(M, J, g)$ is {\it Bismut K\"ahler-like} ({\it BKL} for short) if the torsionless first Bianchi identity \eqref{eq:1bianchi} is satisfied.
\end{definition}

Notice that, as a consequence of Lemma \ref{lem:surprise}, equation \eqref{eq:Rcomplex} is automatically satisfied on a BKL manifold.

\begin{rem}
We observe that the proof of Lemma \ref{lem:surprise} can be carried out for any Hermitian connection. Therefore, the K{\"a}hler-like condition \cite[Definition 4]{MR4564028} for a Hermitian connection reduces to the torsionless first Bianchi identity \eqref{eq:1bianchi} (see also \cite[Remark 3]{MR4564028}).
\end{rem}

Finally, by \cite[Theorem 1]{MR4554474}, it is known that $(M,J,g)$ is BKL if and only if the torsion 3-form $-{\rm d}^c\omega$ of the Bismut connection is ${\rm d}$-closed and $\nabla$-parallel, {\it i.e.},
\begin{equation} \label{eq:BKL}
{\rm d}{\rm d}^c\omega = 0 \,\, , \quad \nabla {\rm d}^c\omega = 0 \,\, .
\end{equation}
Recall that, when ${\rm d}{\rm d}^c\omega = 0$, the metric $g$ is called {\it pluriclosed}, and so the BKL condition is then equivalent to being {\it pluriclosed with parallel Bismut torsion}. Even though the focus of the article is on the conditions \eqref{eq:BKL}, we will address these manifolds as BKL for the sake of simplicity.

\subsection{Some basics on compact Lie algebras} \label{sect:prelLie} \hfill \par

Let $\mathfrak{g}$ be a real Lie algebra and $\mathcal{B}_{\mathfrak{g}}$ its Cartan-Killing form. It is known that the following three conditions are equivalent: \begin{itemize}
\item[$i)$] there exists a ${\rm ad}(\mathfrak{g})$-invariant Euclidean inner product on $\mathfrak{g}$;
\item[$ii)$] $\mathcal{B}_{\mathfrak{g}}$ is non-positive definite;
\item[$iii)$] $\mathfrak{g}$ is {\it compact}, {\it i.e.}, it is the Lie algebra of a compact Lie group.
\end{itemize}
Assume then that $\mathfrak{g}$ is compact. In that case, $\mathfrak{g}$ is {\it reductive}, {\it i.e.}, it splits as $\mathfrak{g} = \mathfrak{z} +\mathfrak{k}$. Here, $\mathfrak{z}$ denotes the center of $\mathfrak{g}$, where $\mathcal{B}_{\mathfrak{g}}$ vanishes, and $\mathfrak{k} \coloneqq [\mathfrak{g},\mathfrak{g}]$ is the semisimple part of $\mathfrak{g}$, where $\mathcal{B}_{\mathfrak{g}}$ is negative definite. We also recall that $\mathfrak{k}$ splits as a $(-\mathcal{B}_{\mathfrak{g}})$-orthogonal sum of simple ideals
\begin{equation} \label{eq:k-simplefactors}
\mathfrak{k} = \mathfrak{k}_1 + {\dots} + \mathfrak{k}_r
\end{equation}
and that, for any $1 \leq i \leq r$, the restriction of $\mathcal{B}_{\mathfrak{g}}$ to $\mathfrak{k}_i$ coincides with the Cartan-Killing form of $\mathfrak{k}_i$, that is, $\mathcal{B}_{\mathfrak{g}}|_{\mathfrak{k}_i \otimes \mathfrak{k}_i} = \mathcal{B}_{\mathfrak{k}_i}$. We call {\it rank of $\mathfrak{g}$} the dimension of a maximal abelian Lie subalgebra of $\mathfrak{g}$.

A {\it linear complex structure} on $\mathfrak{g}$ is an endomorphism $J \in {\rm End}({\mathfrak{g}})$ that satisfies the following two properties: \begin{itemize}
\item[$\bcdot$] $J \circ J = -{\rm Id}_{\mathfrak{g}}$;
\item[$\bcdot$] $[JX,JY] -[X,Y] -J[JX,Y] -J[X,JY] = 0$ for any $X, Y \in \mathfrak{g}$.
\end{itemize}
Samelson proved in \cite{MR0059287} that any compact Lie algebra of even rank admits a linear complex structure. Moreover, by \cite{MR994129}, any linear complex structure on a compact Lie algebra of even rank is obtained by means of Samelson's construction. More precisely, if $\mathfrak{g}$ is a compact Lie algebra of even rank, $g$ an ${\rm ad}(\mathfrak{g})$-invariant Euclidean inner product on $\mathfrak{g}$ and $J$ a $g$-orthogonal linear complex structure on $\mathfrak{g}$, then there exists a unique maximal abelian Lie subalgebra $\tilde{\mathfrak{t}} \subset \mathfrak{g}$ that verifies the following properties: \begin{itemize}
\item[$a)$] $\tilde{\mathfrak{t}}$ is $J$-invariant, {\it i.e.}, $J\tilde{\mathfrak{t}} = \tilde{\mathfrak{t}}$;
\item[$b)$] $J$ is ${\rm ad}(\tilde{\mathfrak{t}})$-invariant, {\it i.e.}, $[V,JX] = J[V,X]$ for any $V \in \tilde{\mathfrak{t}}$, $X \in \mathfrak{g}$.
\end{itemize}
Let now $\mathfrak{t}$ be the $g$-orthogonal complement of $\mathfrak{z}$ inside $\tilde{\mathfrak{t}}$ and denote by $\Delta$ the root system for the pair $(\mathfrak{k},\mathfrak{t})$. By means of the above conditions, it follows that $J$ also determines a choice of positive roots $\Delta^+ \subset \Delta$, that is unique up to the action of the Weyl group. Denote by $\mathfrak{k}_{\alpha}$ the complex $1$-dimensional $\alpha$-root space for any $\alpha \in \Delta$, that is,
$$
\mathfrak{k}_{\alpha} \coloneqq \{ E \in \mathfrak{k} \otimes_{\mathbb{R}} \mathbb{C} : [H,E] = \mathtt{i}\alpha(H)E \text{ for any $H \in \mathfrak{t}$}\} \, ,
$$
and set
$$
\mathfrak{r}_{\alpha} \coloneqq (\mathfrak{k}_{\alpha} +\mathfrak{k}_{-\alpha}) \cap \mathfrak{k} \, , \,\, \text{ for any $\alpha \in \Delta^+$} \, .
$$
Then, the Lie algebra $\mathfrak{g}$ splits $g$-orthogonally as
\begin{equation} \label{eq:split-g}
\mathfrak{g} = \mathfrak{z} + \mathfrak{t} + \sum_{\alpha \in \Delta^+} \mathfrak{r}_{\alpha}
\end{equation}
and $J$ preserves each {\it real root space} $\mathfrak{r}_{\alpha}$. Notice that $g$ leaves the decompositions \eqref{eq:k-simplefactors} orthogonal and that its restriction to any simple ideal $\mathfrak{k}_i$ is of the form $g|_{\mathfrak{k}_i \otimes \mathfrak{k}_i} = \kappa_i (-\mathcal{B}_{\mathfrak{k}_i})$ for some $\kappa_i >0$.

For every $\alpha \in \Delta$ we can fix a generator $E_{\alpha}$ for $\mathfrak{k}_{\alpha}$ as in \cite[Sections III.4 - III.5 - III.6]{MR1834454}. Therefore, we get a $(J,g)$-unitary basis $(X_{\alpha},Y_{\alpha})$ for the real root space $\mathfrak{r}_{\alpha}$ by setting
\begin{equation} \label{eq:basisalpha}
X_{\alpha} \coloneqq \tfrac{1}{\sqrt{2}}(E_{\alpha}-E_{-\alpha}) \,\, , \quad
Y_{\alpha} \coloneqq \tfrac{\mathtt{i}}{\sqrt{2}}(E_{\alpha}+E_{-\alpha}) \quad \text{for all $\alpha \in \Delta^+$} \,\, .
\end{equation}
By \cite[Sections III.4 - III.5 - III.6]{MR1834454} and \eqref{eq:basisalpha}, we get the following bracket relations for $\alpha \neq \beta$:
\begin{equation} \begin{aligned} \label{eq:structcoef1} 
[X_{\alpha},X_{\beta}] &= \eta_{\alpha, \beta}X_{\alpha +\beta} -\sgn(\alpha -\beta)\eta_{\alpha, -\beta}X_{|\alpha -\beta|} \,\, ,\\
[X_{\alpha},Y_{\beta}] &= \eta_{\alpha, \beta}Y_{\alpha +\beta} +\eta_{\alpha, -\beta}Y_{|\alpha -\beta|} \,\, \\
[Y_{\alpha},Y_{\beta}] &= -\eta_{\alpha, \beta}X_{\alpha +\beta} -\sgn(\alpha -\beta)\eta_{\alpha, -\beta}X_{|\alpha -\beta|} \,\, ,
\end{aligned} \end{equation}
where
$$
\sgn(\alpha -\beta) \coloneqq \left\{\!\!\begin{array}{cc}
+1 & \text{if $\alpha -\beta \in \Delta^+$} \\
-1 & \text{if $\beta -\alpha \in \Delta^+$} \\
0 & \text{otherwise}
\end{array}\right. \,\, , \quad
|\alpha -\beta| \coloneqq \sgn(\alpha -\beta)\,(\alpha -\beta) \,\, ,
$$
and the coefficients $\eta_{\alpha, \beta} \in \mathbb{R}$ satisfy
\begin{equation} \begin{gathered} \label{eq:structcoef2} 
\eta_{\alpha, \beta} = 0 \quad \text{if $\alpha +\beta \notin \Delta$} \,\, , \\
\eta_{\alpha, \beta} = -\eta_{\beta, \alpha} \,\, , \quad \eta_{-\alpha, -\beta} = -\eta_{\alpha, \beta} \,\, , \\
\eta_{\alpha, \beta} = \eta_{\beta, -(\alpha +\beta)} = \eta_{-(\alpha +\beta), \alpha} \,\, .
\end{gathered} \end{equation}
Moreover, for any $\alpha \in \Delta^+$, there exists a vector $H_{\alpha} \in \mathfrak{t}$ such that ${\rm span}_{\mathbb{R}}(H_{\alpha}, X_{\alpha}, Y_{\alpha})$ is a Lie subalgebra of $\mathfrak{k}$ isomorphic to $\mathfrak{su}(2)$ and $\mathfrak{t} = {\rm span}_{\mathbb{R}}(H_{\alpha} : \alpha \in \Delta^+)$ (see also \cite[Section 4.3]{MR3362465}). \smallskip

Finally, we recall that a {\it (real) Lie algebra bundle over $M$} is a triple $\mathcal{G} = (\mathcal{G},\mu,\mathfrak{g})$ given by a vector bundle $\pi:\mathcal{G} \to M$, a morphism of vector bundles $\mu: \Lambda^2\mathcal{G} \to \mathcal{G}$ covering the identity and a finite dimensional real Lie algebra $\mathfrak{g}$ satisfying the following properties: \begin{itemize}
\item[$i)$] for any $p \in M$, the linear map $\mu_p$ verifies the Jacobi identity;
\item[$ii)$] for any $p \in M$, there exists a neighborhood $\EuScript{U} \subset M$ of $p$ and a diffeomorphism $\phi: \pi^{-1}(\EuScript{U}) \to \EuScript{U} \times \mathfrak{g}$ trivializing $\mathcal{G}$ such that the restriction to each fiber $\phi_q : (\mathcal{G}_q,\mu_q) \to \mathfrak{g}$ is a Lie algebra isomorphism.
\end{itemize}
By \cite[Theorem 3]{MR0500955}, condition $(ii)$ is automatically satisfied if $(i)$ holds and all the fibers $(\mathcal{G}_q,\mu_q)$ are isomorphic as Lie algebras. A vector subbundle $\mathcal{H} \subset \mathcal{G}$ is called {\it Lie algebra subbundle of $\mathcal{G}$} if it is closed under the morphism $\mu$ and there exists a Lie subalgebra $\mathfrak{h} \subset \mathfrak{g}$ such that $(\mathcal{H},\mu|_{\Lambda^2\mathcal{H}},\mathfrak{h})$ is a Lie algebra bundle.

\medskip
\section{Construction of pluriclosed manifolds with parallel Bismut torsion}
\label{sect:if-part} \setcounter{equation} 0

In this section, we present a general construction of non-K{\"a}hler BKL manifolds in arbitrary dimension. By Theorem \ref{thm:MAIN-A}, it will turn out that any BKL manifold can be obtained as a product of K{\"a}hler manifolds and these building blocks. We begin by recalling the following definition.

\begin{definition}
A {\it Sasaki manifold} is a triple $(S,g,\xi)$ consisting of an odd dimensional Riemannian manifold $(S,g)$, together with a distinguished Killing vector field $\xi$ with unit norm, called {\it Reeb vector field}, such that
\begin{equation} \label{eq:Sasakidef}
\tfrac{1}{c^2}D^2_{X,Y}\xi = g(\xi,Y)X -g(X,Y)\xi \quad \text{for every $X,Y$,}
\end{equation}
for some $c >0$. The distribution $\xi^{\perp} \subset TS$ is called the {\it transverse Sasaki distribution}, while the tensor field $\frac{1}{c}D\xi$ is called the {\it canonical transverse complex structure}.
\end{definition}

We remark that this condition is usually referred to as {\it $c$-Sasaki}, with the Sasaki case corresponding to $c =1$ (see, {\it e.g.} \cite[Proposition 1.2]{MR2893680} and \cite{MR2382957}). This formulation is better suited to the BKL context since it is invariant under rescaling of the metric, unlike the classical Sasaki condition, and it is consistent with the previous literature on BKL manifolds (see \cite{MR4577328,MR4554474,ZZ4-5}). \smallskip

Let $S_{i}$ be a $3$-dimensional manifold endowed with a fixed Sasaki structure, with $1 \leq i \leq s$, and $\mathsf{K}$ a compact semisimple Lie group of ${\rm rank}(\mathsf{K}) = r$ endowed with a bi-invariant metric. Let $\ell \in \mathbb{N}$ be such that $2m \coloneqq \ell + s + r$ is even and define the product
\begin{equation} \label{eq:standardBKL}
M^{2n} \coloneqq \mathbb{R}^{\ell} \times \prod_{i=1}^{s} S_{i} \times \mathsf{K} \,\, ,
\end{equation}
where $\mathbb{R}^{\ell}$ is considered with the Euclidean structure. Denote by $g$ the product Riemannian metric on $M$ and define a complex structure $J$ as follows.

Fix a global $D$-parallel frame $(H_{1},{\dots},H_{\ell})$ for $\mathbb{R}^{\ell}$ and denote by $(H_{\ell +1},{\dots},H_{\ell +s})$ the Reeb vector fields associated to the given Sasaki manifolds.  Let $\mathsf{T}$ be a maximal torus of $\mathsf{K}$, together with a left-invariant orthonormal frame $(H_{\ell +s +1},{\dots},H_{2m})$ tangent to $\mathsf{T}$. Notice that, by construction,
\begin{equation} \label{eq:commHH}
[H_i, H_j] = 0 \quad \text{ for any $1 \leq i, j \leq 2m$} \,\, .
\end{equation}
Denote also by $\varpi: M \to \mathsf{K}/\mathsf{T}$ the composition of the canonical projection of $M$ onto its last factor and the projection $\mathsf{K} \to \mathsf{K}/\mathsf{T}$.

\begin{definition} \label{def:standard}
An almost complex structure $J$ on a Riemannian manifold $(M,g)$ as in \eqref{eq:standardBKL} is called {\it standard} if it satisfies the following properties. \begin{itemize}
\item[$i)$] The distribution spanned by the global orthonormal frame $(H_1,{\dots},H_{2m})$ on $M$ is $J$-invariant, and $J$ acts on these vector fields as
\begin{equation} \label{eq:JH}
JH_i = \sum_{j = 1}^{2m}\,H_j a_{ji} \,\, ,
\end{equation}
where $A = (a_{ji}) \in \mathsf{U}(m) \subset \mathsf{SO}(2m)$ is a constant matrix.
\item[$ii)$] On each transverse Sasaki distribution $H_{\ell +i}^{\perp} \cap T S_i$, with $1 \leq i \leq s$, $J$ coincides with the canonical transverse complex structure of $S_i$.
\item[$iii)$] $J$ projects via $\varpi$ onto a $\mathsf{K}$-invariant complex structure on the full flag manifold $\mathsf{K}/\mathsf{T}$.
\end{itemize}
\end{definition}

Notice that this construction generalizes both the complex structure on the product of two Sasaki manifolds defined in \cite{MR0163246, MR0630631} and Samelson's complex structures \cite{MR0059287}. Moreover, any almost complex structure defined accordingly to Definition \ref{def:standard} turns out to be integrable. Indeed, we have the following lemma.

\begin{lemma}
The triple $(M,g,J)$, where $(M,g)$ is given by \eqref{eq:standardBKL} and $J$ is standard according to Definition \ref{def:standard}, is a Hermitian manifold.
\end{lemma}

\begin{proof}
The fact that $J$ is a $g$-orthogonal almost complex structure follows directly from the construction. Moreover, by properties $ii)$ and $iii)$ in Definition \ref{def:standard}, the only non-trivial components of the Nijenhuis tensor $N_J$ are those given by $N_J(H_i,H_j)$ for $1 \leq i, j \leq 2m$. However, since the coefficients $a_{ij}$ in \eqref{eq:JH} are constant, they vanish by \eqref{eq:commHH}.
\end{proof}

We now prove that these Hermitian manifolds are BKL. Let us consider now the Lie algebra $\mathfrak{k} \coloneqq \Lie(\mathsf{K})$ and define the left-invariant $3$-fom $B \in \Lambda^3\mathfrak{k}^*$ by
\begin{equation} \label{eq:defB}
B(V,W,Z) \coloneqq -g([V,W],Z) \,\, .
\end{equation}

\begin{proposition} \label{prop:T}
Let $(M,g)$ be a Riemannian manifold as in \eqref{eq:standardBKL} endowed with a standard complex structure $J$. Then, the torsion $T$ of its Bismut connection is given by
\begin{equation} \label{eq:Tsplit}
T = \sum_{i=1}^{s} H_{\ell +i}^* \wedge {\rm d}H_{\ell +i}^* + B \,\, ,
\end{equation}
where $H_{\ell +i}^*$ denotes the metric dual $1$-form of $H_{\ell +i}$.
\end{proposition}

\begin{proof}
We begin by computing the fundamental $2$-form $\omega = g(J\cdot\,, \cdot)$ with respect to an appropriate frame. Since the vector fields $(H_{1},{\dots},H_{\ell})$ on $\mathbb{R}^{\ell}$ are $D$-parallel, we have
\begin{equation} \label{eq:diffcofRell}
{\rm d}H_i^* = 0 \quad \text{ for $1 \leq i \leq \ell$} \,\, .
\end{equation}
Let us fix now a local unitary frame $(E_{\ell +i}, F_{\ell +i})$ for $H_{\ell +i}^{\perp} \cap T S_i$, for any $1 \leq i \leq s$. A straightforward computation shows then that
\begin{equation} \label{eq:DHEF} \begin{aligned}
DH_{\ell +i} &= c_i \big( E_{\ell +i}^* \otimes F_{\ell +i} - F_{\ell +i}^* \otimes E_{\ell +i} \big) \,\, , \\
DE_{\ell +i} &= c_i F_{\ell +i}^* \otimes H_{\ell +i} -\vartheta_{\ell +i} \otimes F_{\ell +i} \,\, , \\
DF_{\ell +i} &= -c_i E_{\ell +i}^* \otimes H_{\ell +i} +\vartheta_{\ell +i} \otimes E_{\ell +i} \,\, ,
\end{aligned} \end{equation}
where $\vartheta_{\ell +i}$ is a local $1$-form on $S_i$ and $c_i >0$. Therefore, the differentials of the associated coframe $(H_{\ell +i}^*,E_{\ell +i}^*,F_{\ell +i}^*)$ is given by
\begin{equation} \label{eq:diffcofSas}
\begin{aligned}
{\rm d}H_{\ell +i}^* &= 2c_i E_{\ell +i}^* \wedge F_{\ell +i}^* \,\, , \\
{\rm d}E_{\ell +i}^* &= c_i F_{\ell +i}^* \wedge H_{\ell +i}^* -\vartheta_{\ell +i} \wedge F_{\ell +i}^* \,\, ,\\
{\rm d}F_{\ell +i}^* &= -c_i E_{\ell +i}^* \wedge H_{\ell +i}^* +\vartheta_{\ell +i} \wedge E_{\ell +i}^* \,\, .
\end{aligned} \end{equation}
We observe also that, for any $V \in \mathfrak{k}$, the differential ${\rm d}V^*$ of its metric dual $V^*$ with respect to $g$ verifies
\begin{equation} \label{eq:dVB}
{\rm d}V^* = B(V,\cdot\,,\cdot\,) \,\, .
\end{equation}
By assumption (iii) in Definition \ref{def:standard}, there exists a choice of positive roots $\Delta^+$ for $\mathfrak{k}$, with respect to $\mathfrak{t} \coloneqq \Lie(\mathsf{T})$, such that $J$ restricts to a linear complex structure on each real root space $\mathfrak{r}_{\alpha}$ of $\mathfrak{k}$. Consequently, we can  fix, for each $\alpha \in \Delta^+$, a frame of unitary left-invariant vector fields $(X_{\alpha}, Y_{\alpha})$ for $\mathfrak{r}_{\alpha}$ as in Section \ref{sect:prelLie}. A direct computation based on \eqref{eq:dVB} shows that
\begin{equation} \label{eq:diffcofKh}
{\rm d}H_{\ell +s +i}^* = -\sum_{\alpha \in \Delta^+} \alpha(H_{\ell +s +i})\, X_{\alpha}^* \wedge Y_{\alpha}^* \quad \text{ for $1 \leq i \leq r$} \,\, .
\end{equation}
For the sake of notation, we also set
$$
\alpha(H_i) = 0 \quad \text{ for any $1 \leq i \leq \ell +s$, for any $\alpha \in \Delta^+$} \,\, .
$$

Let now $Q = (q_{ji}) \in \mathsf{SO}(2m)$ be such that
$$
A = Q \left(\!\!\tiny{\begin{array}{ccccc}
0 \!\!\!&\!\!\! -1 \!\!\!&\!\!\! \!\!\!&\!\!\! \!\!\!&\!\!\! \\
1 \!\!\!&\!\!\! 0 \!\!\!&\!\!\! \!\!\!&\!\!\! \!\!\!&\!\!\! \\
 \!\!\!&\!\!\! \!\!\!&\!\!\! {\ddots} \!\!\!&\!\!\! \!\!\!&\!\!\! \\
 \!\!\!&\!\!\! \!\!\!&\!\!\! \!\!\!&\!\!\! 0 \!\!\!&\!\!\! -1 \\
 \!\!\!&\!\!\! \!\!\!&\!\!\! \!\!\!&\!\!\! 1 \!\!\!&\!\!\! 0 \\
\end{array}}\!\!\right) Q^t \,\, ,
$$
where $A = (a_{ji}) \in \mathsf{U}(m)$ is as in \eqref{eq:JH}, and denote by $(Z_1,{\dots},Z_{2m})$ the new frame defined by
\begin{equation} \label{def:Z}
Z_i \coloneqq \sum_{j = 1}^{2m}\,H_j q_{ji} \,\, ,
\end{equation}
so that $g(Z_i,Z_j) = \delta_{ij}$ and $JZ_{2i-1} = Z_{2i}$. Therefore, $\omega$ takes the form
\begin{equation} \label{eq:dcw0}
\omega = \sum_{i=1}^m Z_{2i-1}^* \wedge Z_{2i}^* +\sum_{i=1}^s E_{\ell +i}^* \wedge F_{\ell +i}^* +\sum_{\alpha \in \Delta^+} X_{\alpha}^* \wedge Y_{\alpha}^* \,\, .
\end{equation}
Let us compute now the $3$-form $T = -{\rm d}^c\omega$. By \eqref{eq:diffcofSas}, it is immediate to see that
$$
{\rm d}(E_{\ell +i}^* \wedge F_{\ell +i}^*) = 0 \quad \text{for any $1 \leq i \leq s$}
$$
and so
\begin{equation} \label{eq:dcw1}
-{\rm d}^c\left(\sum_{i=1}^m E_{\ell +i}^*\wedge F_{\ell +i}^*\right) = 0 \,\, .
\end{equation}
Moreover, by \eqref{eq:diffcofKh} and \eqref{def:Z}, one has
$$\begin{aligned}
J^{-1}{\rm d}Z_i^* &= J^{-1}\left(\sum_{j=1}^{2m} {\rm d}H_j^* q_{ji}\right) \\
&=J^{-1}\left(\sum_{j=\ell +1}^{\ell +s} 2c_jE_j^*\wedge F_j^* q_{ji}
-\sum_{j=\ell +s +1}^{2m} \sum_{\alpha \in \Delta^+} \alpha(H_j)q_{ji} X_{\alpha}^* \wedge Y_{\alpha}^*\right) \\
&= \sum_{j=\ell +1}^{\ell +s} 2c_jE_j^*\wedge F_j^* q_{ji} -\sum_{j=\ell +s +1}^{2m} \sum_{\alpha \in \Delta^+} \alpha(H_j)q_{ji} X_{\alpha}^* \wedge Y_{\alpha}^*
\end{aligned}$$
and so
\begin{align}
-{\rm d}^c\left(\sum_{i=1}^m Z_{2i-1}^*\wedge Z_{2i}^*\right) &= \sum_{i=1}^m \big(-{\rm d}^cZ_{2i-1}^*\wedge Z_{2i}^* +Z_{2i-1}^*\wedge {\rm d}^cZ_{2i}^* \big) \nonumber \\
&= \sum_{i=1}^{2m} Z_{i}^* \wedge J^{-1}{\rm d}Z_{i}^* \nonumber \\
&= \sum_{i=1}^{2m} Z_{i}^* \wedge \left(\sum_{j=\ell +1}^{\ell +s} 2c_jE_j^*\wedge F_j^* q_{ji}
-\sum_{j=\ell +s +1}^{2m} \sum_{\alpha \in \Delta^+} \alpha(H_j)q_{ji} X_{\alpha}^* \wedge Y_{\alpha}^*\right) \nonumber \\
&= \sum_{j=\ell +1}^{\ell +s} 2c_j H_j^* \wedge E_j^*\wedge F_j^*
-\sum_{j=\ell +s +1}^{2m} \sum_{\alpha \in \Delta^+} \alpha(H_j) H_j^* \wedge X_{\alpha}^* \wedge Y_{\alpha}^* \nonumber \\
&= \sum_{i=1}^{s} H_{\ell +i}^* \wedge {\rm d}H_{\ell +i}^* +\sum_{i=1}^{r} H_{\ell +s +i}^* \wedge {\rm d}H_{\ell +s +i}^* \,\, . \label{eq:dcw2}
\end{align}
For the last term of $T$, we observe that
$$
{\rm d}^cX_\alpha^* = -J^{-1}{\rm d}Y_\alpha^* \,\, , \quad {\rm d}^cY_\alpha^* = J^{-1}{\rm d}X_\alpha^*
$$
and so
\begin{equation} \label{eq:dcw3}
-{\rm d}^c\left(\sum_{\alpha \in \Delta^+} X_{\alpha}^*\wedge Y_{\alpha}^*\right) = \sum_{\alpha \in \Delta^+} X_{\alpha}^* \wedge J^{-1}{\rm d}X_{\alpha}^* +Y_{\alpha}^* \wedge J^{-1}{\rm d}Y_{\alpha}^* \,\, .
\end{equation}
Is it then easy to see that
\begin{equation} \label{eq:dcw4}
E_{\ell +i} \,\lrcorner\, {\rm d}^c\left(\sum_{\alpha \in \Delta^+} X_{\alpha}^*\wedge Y_{\alpha}^*\right) = F_{\ell +i} \,\lrcorner\, {\rm d}^c\left(\sum_{\alpha \in \Delta^+} X_{\alpha}^*\wedge Y_{\alpha}^*\right) = 0 \,\, .
\end{equation}
Moreover, a straightforward computation shows that, for any $1 \leq i \leq 2m$,
$$
H_i \,\lrcorner\, J^{-1}{\rm d}X_{\alpha}^* = \alpha(H_i)\, X_{\alpha}^* \,\, , \quad H_i \,\lrcorner\, J^{-1}{\rm d}Y_{\alpha}^* = \alpha(H_i)\, Y_{\alpha}^*
$$
and so, by \eqref{eq:dcw3}, we get
\begin{align}
H_i \,\lrcorner\, {\rm d}^c\left(\sum_{\alpha \in \Delta^+} X_{\alpha}^*\wedge Y_{\alpha}^*\right) &= -\sum_{\alpha \in \Delta^+} H_i \,\lrcorner\, \Big(X_{\alpha}^* \wedge J^{-1}{\rm d}X_{\alpha}^*\Big) +H_i \,\lrcorner\, \Big(Y_{\alpha}^* \wedge J^{-1}{\rm d}Y_{\alpha}^*\Big) \nonumber \\
&= \sum_{\alpha \in \Delta^+} X_{\alpha}^* \wedge \Big(H_i \,\lrcorner\, J^{-1}{\rm d}X_{\alpha}^*\Big) + Y_{\alpha}^* \wedge \Big(H_i \,\lrcorner\, J^{-1}{\rm d}Y_{\alpha}^*\Big) \nonumber \\
&= \sum_{\alpha \in \Delta^+} \alpha(H_i) \big( X_{\alpha}^* \wedge X_{\alpha}^*+ Y_{\alpha}^* \wedge Y_{\alpha}^* \big) \nonumber \\
&= 0 \,\, . \label{eq:dcw5}
\end{align}
Finally we observe that, by \eqref{eq:diffcofKh}, it follows that
$$
{\rm d}H_{\ell +s +i}^* = J^{-1}{\rm d}H_{\ell +s +i}^* \quad \text{for any $1 \leq i \leq r$}
$$
and so
\begin{align}
&\left(\sum_{i=1}^{r} H_{\ell +s +i}^* \wedge {\rm d}H_{\ell +s +i}^* +\sum_{\alpha \in \Delta^+} X_{\alpha}^* \wedge J^{-1}{\rm d}X_{\alpha}^* +Y_{\alpha}^* \wedge J^{-1}{\rm d}Y_{\alpha}^*\right)(U,V,W) = \nonumber \\
&\quad = \left(\sum_{i=1}^{r} H_{\ell +s +i}^* \wedge J^{-1}{\rm d}H_{\ell +s +i}^* +\sum_{\alpha \in \Delta^+} X_{\alpha}^* \wedge J^{-1}{\rm d}X_{\alpha}^* +Y_{\alpha}^* \wedge J^{-1}{\rm d}Y_{\alpha}^*\right)(U,V,W) \nonumber \\
&\quad = \sum_{i=1}^r \Big(
g(H_{\ell +s +i},U) {\rm d}H_{\ell +s +i}^* (JV,JW)
+g(H_{\ell +s +i},V) {\rm d}H_{\ell +s +i}^* (JW,JU) \nonumber \\
&\quad\quad +g(H_{\ell +s +i},W) {\rm d}H_{\ell +s +i}^* (JU,JV) \Big) \nonumber \\
&\quad\quad +\sum_{\alpha \in \Delta^+} \Big(
g(X_{\alpha},U) {\rm d}X_{\alpha}^* (JV,JW)
+g(X_{\alpha},V) {\rm d}X_{\alpha}^* (JW,JU)
+g(X_{\alpha},W) {\rm d}X_{\alpha}^* (JU,JV) \nonumber \\
&\quad\quad +g(Y_{\alpha},U) {\rm d}Y_{\alpha}^* (JV,JW)
+g(Y_{\alpha},V) {\rm d}Y_{\alpha}^* (JW,JU)
+g(Y_{\alpha},W) {\rm d}Y_{\alpha}^* (JU,JV) \Big) \nonumber \\
&\quad = \sum_{i=1}^r \Big(
g(H_{\ell +s +i},U) B(H_{\ell +s +i},JV,JW)
+g(H_{\ell +s +i},V) B(H_{\ell +s +i},JW,JU) \nonumber \\
&\quad\quad +g(H_{\ell +s +i},W) B(H_{\ell +s +i},JU,JV) \Big) \nonumber \\
&\quad\quad +\sum_{\alpha \in \Delta^+} \Big(
g(X_{\alpha},U) B(X_{\alpha},JV,JW)
+g(X_{\alpha},V) B(X_{\alpha},JW,JU)
+g(X_{\alpha},W) B(X_{\alpha},JU,JV) \nonumber \\
&\quad\quad +g(Y_{\alpha},U) B(Y_{\alpha},JV,JW)
+g(Y_{\alpha},V) B(Y_{\alpha},JW,JU)
+g(Y_{\alpha},W) B(Y_{\alpha},JU,JV) \Big) \nonumber \\
&\quad= B(U,JV,JW) +B(V,JW,JU) +B(W,JU,JV) \nonumber \\
&\quad= -g(U,[JV,JW]) +g(U,J[V,JW]) +g(U,J[JV,W]) \nonumber \\
&\quad= -g(U,[V,W]) \nonumber \\
&\quad=B(U,V,W) \,\, . \label{eq:dcw6}
\end{align}
Therefore, \eqref{eq:Tsplit} follows from \eqref{eq:dcw0}, \eqref{eq:dcw1}, \eqref{eq:dcw2}, \eqref{eq:dcw3}, \eqref{eq:dcw4}, \eqref{eq:dcw5}, and \eqref{eq:dcw6}.
\end{proof}

We are ready now to prove the main result of this section.

\begin{theorem} \label{thm:standard-BKL}
Any Riemannian manifold $(M,g)$ as in \eqref{eq:standardBKL} endowed with a standard complex structure $J$ is BKL.
\end{theorem} 

\begin{proof}
By \eqref{eq:Tsplit}, the torsion $3$-form $T$ splits accordingly to \eqref{eq:standardBKL}, and so $\nabla$ restricts to metric connections with totally skew-symmetric torsion on each factor of \eqref{eq:standardBKL}. By \cite[Theorem 1]{MR4554474}, the BKL condition is equivalent to \eqref{eq:BKL}, and so we need to show that ${\rm d}T=0$ and $\nabla T=0$ on each factor.

On $\mathsf{K}$, the restricted Bismut connection takes the form
$$
\nabla = D +\tfrac12B \,\, ,
$$
where $B$ was defined in \eqref{eq:defB}, and so is flat by \cite[Chapter X, Proposition 2.12]{MR1393941}. A direct computation based on the Jacobi identity shows that ${\rm d}B = 0$. Moreover, by \cite[Proposition 2.1]{MR2651537}, it follows that
$$
(\nabla_VB)(X,Y,Z) = \tfrac12{\rm d}B(X,Y,Z,V) = 0
$$
for any $X,Y,Z,V \in \mathfrak{k}$.

On each $S_i$, one has
$$
{\rm d}\big(H_{\ell +i}^* \wedge {\rm d}H_{\ell +i}^*\big) = 0 \,\, .
$$
Moreover, if we choose a local unitary frame $(E_{\ell +i}, F_{\ell +i})$ for $H_{\ell +i}^{\perp} \cap TS_i$ as in the proof of Proposition \ref{prop:T}, by \eqref{eq:DHEF} and \eqref{eq:diffcofSas} we get
$$\begin{aligned}
\nabla H_{\ell +i} &= 0 \,\, , \\
\nabla E_{\ell +i} &= c_iH_{\ell +i}^* \otimes F_{\ell +i} -\vartheta_{\ell +i} \otimes F_{\ell +i} \,\, , \\
\nabla F_{\ell +i} &= -c_iH_{\ell +i}^* \otimes E_{\ell +i} +\vartheta_{\ell +i} \otimes E_{\ell +i} \,\, ,
\end{aligned}$$
and so
$$
\nabla H_{\ell +i}^* =0 \,\, , \quad \nabla{\rm d}H_{\ell +i}^* =  2c_i \nabla\big(E_{\ell +i}^* \wedge F_{\ell +i}^*\big) = 0 \,\, .
$$
This concludes the proof.
\end{proof}

\medskip
\section{Classification of pluriclosed manifolds with parallel Bismut torsion}
\label{sect:classification} \setcounter{equation} 0

Let $(M^{2n}, J, g)$ be a simply-connected BKL manifold of complex dimension $n$. We assume that $(M, J)$ is complex irreducible, {\it i.e.}, it does not split as a product of complex manifolds of lower dimension, and that $g$ is non-K\"ahler. We follow the notation of Section \ref{sect:prel}. From equations \eqref{eq:BKL}, the following crucial observation holds true.

\begin{proposition}
Let $o \in M$ be a distinguished point and $\mathfrak{g} \coloneqq (T_oM,T_o)$. Then, $\mathfrak{g}$ is a compact Lie algebra and the triple $\mathcal{G} = (TM, T,\mathfrak{g})$ is a Lie algebra bundle over $M$. Moreover, for any point $p \in M$, $g_p$ is a bi-invariant Euclidean scalar product on $(T_pM,T_p)$ and $J_p$ is a linear complex structure on $(T_pM,T_p)$.
\end{proposition}

\begin{proof}
By \cite[Eq.\ (3.20)]{MR1822270}, equations \eqref{eq:BKL} imply that
$$
g(T(X,Y),T(Z,V)) +g(T(Y,Z),T(X,V)) +g(T(Z,X),T(Y,V)) = 0 \,\, .
$$
Since $T$ is totally skew-symmetric, it follows that
$$
g(T(X,Y),T(Z,V)) = g(T(T(X,Y),Z),V)
$$
and so
$$
T(T(X,Y),Z) +T(T(Y,Z),X) +T(T(Z,X),Y)=0 \,\, ,
$$
{\it i.e.}, $T$ verifies the Jacobi identity. Fix now $p\in M$ and let $\gamma : [0,1] \to M$ be a geodesic verifying $\gamma(0)=o$ and $\gamma(1)=p$. Denote by $\tau_{\gamma} : T_oM \to T_pM$ the $\nabla$-parallel transport along $\gamma$. Then, since $T$ is parallel, it is immediate to realize that $(\tau_{\gamma})^*T_p = T_o$. This implies that each fiber $(T_pM,T_p)$ is isomorphic as Lie algebra to $\mathfrak{g}$, which implies that $\mathcal{G} = (TM, T,\mathfrak{g})$ is a Lie algebra bundle in virtue of \cite[Theorem 3]{MR0500955}.

Fix now $p \in M$ and let $v, w_1, w_2 \in T_pM$. Notice that
$$
g_p(T_p(v,w_1),w_2) = -g_p(w_1,T_p(v,w_2)) \,\, ,
$$
which in turn implies that $g_p$ is a bi-invariant Euclidean scalar product on $(T_pM,T_p)$. Finally, \eqref{eq:dwJ} implies that
$$
T_p(J_pw_1,J_pw_2) = T_p(w_1,w_2) +J_pT_p(J_pw_1,w_2) +J_pT_p(w_1,J_pw_2)
$$
and so $J_p$ is a linear complex structure on $(T_pM,T_p)$.
\end{proof}

By \cite[Section 2]{MR994129}, for any $p \in M$, the complex structure $J_p$ determines a unique maximal torus $\widetilde{\mathcal{T}}_p$ of $(T_pM,T_p)$ (see Section \ref{sect:prelLie}). Let also $\mathcal{Z}_p$ be the center of $(T_pM,T_p)$, $\mathcal{T}_p$ the $g_p$-orthogonal complement of $\mathcal{Z}_p$ in $\widetilde{\mathcal{T}}_p$ and $\mathcal{P}_p$ the $g_p$-orthogonal complement of $\widetilde{\mathcal{T}}_p$ in $T_pM$. In all these cases, the dependence on the point $p$ is smooth and so these spaces determine the following splitting of the Lie algebra bundle $\mathcal{G}$ as the direct sum of $g$-orthogonal subbundles:
\begin{equation} \label{eq:Liealgebroid}
\mathcal{G} = \mathcal{Z} + \mathcal{T} + \mathcal{P} \,\, .
\end{equation}
Notice that $\mathcal{Z}$ and $\mathcal{T}$ are, in particular, Lie algebra subbundles of $\mathcal{G}$, with typical fibers $\mathfrak{z}$ and $\mathfrak{t}$, respectively, while $\mathcal{P}$ is just a vector subbundle, with typical fiber $\mathfrak{p}$. For the sake of notation, we set $\ell \coloneqq {\rm dim}(\mathfrak{z})$ and $q \coloneqq {\rm dim}(\mathfrak{t})$.

\begin{proposition} \label{prop:par}
The distributions $\mathcal{Z}$, $\mathcal{T}$ and $\mathcal{P}$ in \eqref{eq:Liealgebroid} are $\nabla$-parallel.
\end{proposition}

\begin{proof}
Since the parallel transport preserves the Lie algebra structure of the fibers of $\mathcal{G}$, it follows by construction that the distribution $\widetilde{\mathcal{T}} = \mathcal{Z} + \mathcal{T}$ is $\nabla$-parallel. Moreover, the distribution $\mathcal{Z}$ is also $\nabla$-parallel. Indeed, if we fix a point $p \in M$ and a path $\gamma : I \subset \mathbb{R} \to M$ such that $0 \in I$ and $\gamma(0)=p$. Then, given $z \in \mathcal{Z}_p$, we denote by $Z=Z(t)$ the $\nabla$-parallel vector field along $\gamma=\gamma(t)$ such that $Z(0)=z$. Since the torsion $T$ is $\nabla$-parallel, it follows that
$$
\nabla_{\dot{\gamma}(t)}\big(Z(t) \lrcorner\, T_{\gamma(t)}\big) = \big(\nabla_{\dot{\gamma}(t)}Z(t)\big) \lrcorner\, T_{\gamma(t)} +Z(t) \lrcorner\big(\nabla_{\dot{\gamma}(t)}T_{\gamma(t)}\big) = 0
$$
and so, since $z \lrcorner\,T_p = 0$, we get $Z(t) \lrcorner\, T_{\gamma(t)} = 0$ for any $t \in I$. Finally, since $\nabla$ preserves $g$, it follows that $\mathcal{T}$, that is the orthogonal complement of $\mathcal{Z}$ inside $\widetilde{\mathcal{T}}$, and $\mathcal{P}$, that is the orthogonal complement of $\widetilde{\mathcal{T}}$ inside $\mathcal{G}$, are $\nabla$-parallel too.
\end{proof}

Next, we are going to show that $\mathcal{T}$ is $\nabla$-flat, which will allow us to deduce that it is globally generated by $\nabla$-parallel vector fields. We begin with the following

\begin{lemma} \label{lem:tjt}
The following equality holds true: $\mathcal{T} + J\mathcal{T} =  \mathcal{T} +\mathcal{Z}$.
\end{lemma}
\begin{proof}
First of all, we need to prove that $\mathcal{Z}$ does not contain any $J$-invariant subdistribution. As a matter of fact, suppose $\mathcal{V} \subset \mathcal{Z}$ is the maximal $J$-invariant subdistribution of $\mathcal{Z}$, and let $\widetilde{\mathcal{V}}$ be the distribution defined at any point $p \in M$ by
$$
\widetilde{\mathcal{V}}_p \coloneqq \mathcal{V}_p +{\rm span}_{\mathbb{R}}\!\left\{ \nabla^k_{X_1,{\dots},X_k}Y\big|_p : X_1, {\dots},X_k \in \Gamma(\EuScript{U};TM) \, , \,\, Y \in \Gamma(\EuScript{U};\mathcal{V}) \, , \,\, p \in \EuScript{U} \subset M \, , \,\, k \in \mathbb{N} \right\} \,\, .
$$
By construction, since both $J$ and $\mathcal{Z}$ are $\nabla$-parallel, it is straightforward to verify that $\widetilde{\mathcal{V}}$ is a $J$-invariant and $\nabla$-parallel subdistribution of $\mathcal{Z}$. By the maximality assumption, we get $\mathcal{V} = \widetilde{\mathcal{V}}$ and so $\mathcal{V}$ is $\nabla$-parallel. Since $\mathcal{V} \subset \mathcal{Z}$, we get $DX = \nabla X$ for any $X \in \Gamma(\EuScript{U};\mathcal{V})$, and so $\mathcal{V}$ is $D$-parallel. Since, by hypothesis, $(M, J, g)$ does not contain any K\"ahler de Rham factor, we conclude that $\mathcal{V}$ is trivial.
	
Finally, let us consider the sum $\mathcal{T} + J\mathcal{T}$. This is a $J$-invariant distribution in $\mathcal{Z} + \mathcal{T}$, which is also $J$-invariant. Therefore, the orthogonal complement of $\mathcal{T} + J\mathcal{T}$ inside $\mathcal{Z} + \mathcal{T}$ is a $J$-invariant subdistribution of $\mathcal{Z}$. In virtue of the above argument, this complement has to be trivial, and so the thesis follows.
\end{proof}
	
\begin{rem} \label{rem: center dimension}
From the above proof, it immediately follows that $q \geq \ell$, that is, the dimension $\ell ={\rm dim}(\mathfrak{z})$ of the center $\mathfrak{z}$ can not be greater than half of the dimension $\ell + q$ of the whole maximal torus $\mathfrak{t} +\mathfrak{z}$.
\end{rem}

\begin{proposition} \label{prop:flatwhite}
The Bismut curvature $R$ satisfies the following property at any point $p \in M$: 
$$
\text{$R_{v,w}z = 0$ for $v,w,z \in T_pM$ whenever at least one of them lies in $\widetilde{\mathcal{T}}_p$} \,\, .
$$
In particular, $R$ vanishes identically on $\widetilde{\mathcal{T}}$.  
\end{proposition}
\begin{proof}
Fix $p \in M$ and $v,w,z \in T_pM$. We emphasize the following facts.
\begin{itemize}
\item[$i)$] By Lemma \ref{lem:tjt}, any vector $c \in \mathcal{Z}_p$ can be written as a sum $c = h + Jh'$, for some $h, h' \in \mathcal{T}_p$.
\item[$ii)$] Any vector $h \in \mathcal{T}_p$ can be written as
$$
h = \sum_{i} \lambda_i\,T(e_i,Je_i) \,\, ,
$$
for some $e_i \in \mathcal{P}_p$ and $\lambda_i \in \mathbb{R}$ (see Section \ref{sect:prelLie}).
\item[$iii)$] Since $\nabla J = 0$, then $R_{\,\cdot,\cdot} \circ J = J \circ R_{\,\cdot,\cdot}$.
\item[$iv)$] Since $\nabla T = 0$, then
$$
R_{\,\cdot,\cdot}(T(u_1, u_2)) = T(R_{\,\cdot,\cdot}u_1, u_2) + T(u_1,R_{\,\cdot,\cdot}u_2)
$$
for any $u_1,u_2 \in T_pM$.
\end{itemize}
By Proposition \ref{prop:par}, the distributions $\mathcal{Z}$, $\mathcal{T}$ and $\mathcal{P}$ are $\nabla$-parallel, and so if $z \in \mathcal{D}_p$, with $\mathcal{D}_p \in \big\{\mathcal{Z}_p, \mathcal{T}_p, \mathcal{P}_p\}$, then $R_{v,w}z \in \mathcal{D}_p$ for all $v, w \in T_pM$. Therefore, the torsionless first Bianchi identity \eqref{eq:1bianchi} implies that
\begin{equation} \label{eq:pippo1}
R_{v,w}z = 0 \quad \text{if $z \in \mathcal{D}_p$ and $v, w \in \mathcal{D}_p^{\perp}$} \,\, ,
\end{equation}
while \eqref{eq:symR} implies that
\begin{equation} \label{eq:pippo2}
R_{v,w}z = 0 \quad \text{if $v \in \mathcal{D}_p$ and $w \in \mathcal{D}_p^{\perp}$, $z \in T_pM$} \,\, .
\end{equation}
Assume that at least one of $v, w, z$ lie in $\widetilde{\mathcal{T}}_p$. By \eqref{eq:pippo1} and \eqref{eq:pippo2}, it follows that $R_{v,w}z = 0$ unless $v, w, z \in \mathcal{T}_p$ or $v, w, z \in \mathcal{Z}_p$. If $v, w, z \in \mathcal{T}_p$, claims $ii)$, $iii)$ and $iv)$ show that
$$
R_{vw}z = \sum_i \lambda_i \Big(T(R_{v,w}e_i,Je_i) +T(e_i,JR_{v,w}e_i)\Big) \,\, ,
$$
which is zero by \eqref{eq:pippo2}. If $v, w, z \in \mathcal{Z}_p$, claims $i)$ and $iii)$ show that
$$
R_{v,w}z = R_{v,w}h +JR_{v,w}h' \,\, ,
$$
which is zero by \eqref{eq:pippo2}.
\end{proof}

Since $M$ is simply-connected, Proposition \ref{prop:flatwhite} ensures that we can use parallel transport to define a global, orthonormal, $\nabla$-parallel frame $(H_{\ell +1},{\dots},H_{\ell +q})$ for $\mathcal{T}$. Consequently, the skew-symmetric endomorphism fields
$$
H_{\ell +i} \lrcorner\,T\, |_{\mathcal{P}} \quad \text{ with $1 \leq i \leq q$}
$$
are also $\nabla$-parallel, which in turn implies, by using the parallel transport, that their (complexified) eigenspaces are $\nabla$-parallel, and their eigenvalues are constant. 
This allows us to split $\mathcal{P}$ into a sum of globally defined, two-dimensional, $J$-invariant, $\nabla$-parallel distributions $\mathcal{R}_{\alpha}$, corresponding to a choice of positive roots $\alpha \in \Delta^+$ for the Lie algebra $\mathfrak{g}$. More concretely:
\begin{equation} \label{eq:quasifinaldec}
\mathcal{G} = \mathcal{Z} + \mathcal{T} + \sum_{\alpha \in \Delta^+} \mathcal{R}_\alpha \,\, .
\end{equation}

\begin{lemma} \label{lem:almost-flat}
The only possibly non-vanishing components of the curvature $R$ are of the form
$$
R(v_\alpha, Jv_\alpha, v_\alpha, Jv_\alpha) \quad \text{for $\alpha \in \Delta^+$, $v_\alpha \in \mathcal{R}_\alpha|_p$, $p \in M$.}
$$
\end{lemma}

\begin{proof}
Notice that since the distributions $\mathcal{R}_{\alpha}$ are $\nabla$-parallel. Hence, we obtain
$$
R_{\,\cdot,\cdot}\mathcal{R}_\alpha \subset \mathcal{R}_\alpha \,\, .
$$
Moreover, thanks to \eqref{eq:symR}, this symmetry also holds true in the first two entries of $R$, that is, the only possibly non-vanishing components of the curvature $R$ are of the form
\begin{equation} \label{eq:almost-flat1}
R(v_\alpha, Jv_\alpha, v_\beta, Jv_\beta) \quad \text{ for $\alpha, \beta \in \Delta^+$, $v_\alpha \in \mathcal{R}_\alpha|_p$, $v_\beta \in \mathcal{R}_\beta|_p$, $p \in M$. }
\end{equation}
Moreover, using the torsionless first Bianchi identity \eqref{eq:1bianchi} as in the proof of Proposition \ref{prop:flatwhite}, we deduce that the components in \eqref{eq:almost-flat1} with $\alpha \neq \beta$ vanish, and this concludes the proof.
\end{proof}

Lemma \ref{lem:almost-flat}  plays a central role in the proof of the following 

\begin{prop} \label{prop:trick}
Any Lie algebra subbundle of $\mathcal{T} +\mathcal{P}$ whose typical fiber is an ideal of $\mathfrak{t} +\mathfrak{p}$ of rank at least 2 is $\nabla$-flat. 
\end{prop}

\begin{proof}
We fix an open set $\EuScript{U} \subset M$ such that each $\mathcal{R}_{\alpha}|_{\EuScript{U}}$ is generated by a local unitary frame $\{V_{\alpha}, JV_{\alpha}\}$, so that
$$
T(H,V_{\alpha}) = \alpha(H)JV_{\alpha} \,\, , \quad T(H,JV_{\alpha}) = -\alpha(H)V_{\alpha}
$$
for any $H \in \Gamma(\EuScript{U};\mathcal{T})$. 
Let $\theta$ be the connection $1$-form associated to $\nabla$ with respect to the local frame $\{H_{\ell +1}, {\dots}, H_{\ell +q}\} \cup \{V_{\alpha}, JV_{\alpha} : \alpha \in \Delta^+\}$. By construction, $\nabla H_{\ell +i} = 0$ for any $1 \leq i \leq q$. Moreover, we know that $\nabla_X V_{\alpha}$ lies in $\mathcal{R}_{\alpha}|_{\EuScript{U}}$ and it is straightforward to check that it is orthogonal to $V_{\alpha}$, for any $X \in \Gamma(\EuScript{U};TM)$. Therefore, the only possibly non-vanishing components of $\theta$ are given by
\begin{equation} \label{eq:nabladiag}
\nabla V_{\alpha} = \theta_{\alpha} \otimes JV_{\alpha} \,\, .
\end{equation}

Take now two roots $\alpha_1, \alpha_2 \in \Delta^+$ such that their sum is itself a root $\alpha_3 \coloneqq \alpha_1 +\alpha_2 \in \Delta^+$. Then, by standard root spaces relations (see Section \ref{sect:prelLie}), we have
\begin{align*}
2\,{\eta}_{\alpha_1, \alpha_2} \nabla V_{\alpha_3} &= \nabla \big(T(V_{\alpha_1}, V_{\alpha_2})\big) -\nabla \big(T(JV_{\alpha_1}, JV_{\alpha_2})\big) \\
&= T(\nabla V_{\alpha_1}, V_{\alpha_2}) + T(V_{\alpha_1},\nabla V_{\alpha_2}) -T(J\nabla V_{\alpha_1}, JV_{\alpha_2}) -T(JV_{\alpha_1}, J\nabla V_{\alpha_2}) \\
&= (\theta_{\alpha_1} +\theta_{\alpha_2}) \otimes \big(T(V_{\alpha_1}, JV_{\alpha_2}) -T(V_{\alpha_2}, JV_{\alpha_1})\big) \\
& = 2\,{\eta}_{\alpha_1, \alpha_2}\, (\theta_{\alpha_1} +\theta_{\alpha_2}) \otimes JV_{\alpha_3}
\end{align*}
which implies
\begin{equation} \label{eq:thetasum}
\theta_{\alpha_3} = \theta_{\alpha_1} +\theta_{\alpha_2} \,\, .
\end{equation}

Notice now that the curvature tensor $R$ can be locally computed by means of the connection form $\theta$ using the Cartan structure equation $R_{\,\cdot,\cdot} = {\rm d}\theta - \theta\wedge\theta$. Moreover, by \eqref{eq:nabladiag}, it is straightforward to check that the above formula reduces to
\begin{equation} \label{eq:Rdtheta}
R_{\,\cdot,\cdot}V_{\alpha} = {\rm d}\theta_{\alpha} \otimes JV_{\alpha} \,\, .
\end{equation}
Consequently, by using also \eqref{eq:thetasum}, we get
\begin{align*}
R_{\,\cdot,\cdot} V_{\alpha_3} = {\rm d}\theta_{\alpha_1} \otimes JV_{\alpha_3} +{\rm d}\theta_{\alpha_2} \otimes JV_{\alpha_3}
\end{align*}
which, together with Lemma \ref{lem:almost-flat} , implies that the curvature vanishes in these directions. We finally remark that, in any compact simple Lie algebra of rank at least $2$, any positive root is the sum or the difference of two positive roots and so fits in a triple like the one above, concluding the proof.
\end{proof}

It follows from Proposition \ref{prop:trick} that $\mathcal{T} +\mathcal{P}$ decomposes as a sum of $s$ non-Bismut flat Lie algebra subbundles $\mathcal{A}_1^{(i)}$, whose typical fiber is $\mathfrak{su}(2)$, and a Bismut flat Lie algebra subbundle $\mathcal{K}$ with typical fiber $\mathfrak{k}$, that is
\begin{equation} \label{eq:finaldec}
\mathcal{G} = \mathcal{Z} + \sum_{i=1}^{s} \mathcal{A}_1^{(i)} + \mathcal{K} \; .
\end{equation}

By construction, the torsion $T$ splits accordingly to \eqref{eq:finaldec}. Moreover, since the decomposition \eqref{eq:finaldec} is $\nabla$-parallel, it is also $D$-parallel, and hence induces different de Rham factors. Notice that the distribution $\mathcal{Z}$ is $D$-flat, while the other factors are characterized as follows.

\begin{lemma} \label{lem:sasaki}
Any distribution $\mathcal{A}_1^{(i)}$, for $i=1,\ldots,s$, corresponds to a de Rham factor $S_i$ which is a $3$-dimensional Sasaki manifold.
\end{lemma}
\begin{proof} Fix an index $i=1,\ldots,s$. By construction, $S_i$ is a complete Riemannian manifold of dimension $3$ endowed with a distinguished complete, $\nabla$-parallel vector field $H_{\ell +i}$. Since $\nabla_{\cdot} H_{\ell +i} = 0$, it is straightforward to check that it is Killing and
\begin{equation} \label{eq:DHsec4}
D_{\cdot\,} H_{\ell +i} = \frac12\, T(H_{\ell +i},{\cdot}\,) \,\, ,
\end{equation}
which in turn implies that
\begin{equation} \label{eq:DHsec4'}
D^2_{\cdot,\cdot} H_{\ell +i} = \frac{1}{4}T(T(H_{\ell +i},\cdot),\cdot) \,\, .
\end{equation}
Moreover, since $\mathcal{A}_1^{(i)}$ is a Lie algebra bundle whose typical fiber is $\mathfrak{su}(2)$, it follows that $H_{\ell +i}^{\perp} \cap \mathcal{A}_1^{(i)}$ admits local unitary generators $(E_i, F_i)$ satisfying
\begin{equation} \label{eq:THsec4}
T(H_{\ell +i},E_i) = -2c_iF_i \,\, , \quad T(F_i,H_{\ell +i}) = -2c_iE_i \,\, , \quad T(E_i,F_i) = -2c_iH_{\ell +i} \,\, ,
\end{equation}
for some $c_i >0$. Therefore, \eqref{eq:DHsec4'} and \eqref{eq:THsec4} imply that $H_{\ell +i}$ verifies \eqref{eq:Sasakidef}, and so $S_i$ is a Sasaki manifold with Reeb vector field $H_{\ell +i}$.
\end{proof}

\begin{lemma} \label{lem:Sam}
The distribution $\mathcal{K}$ corresponds to a compact semisimple Lie group $\mathsf{K}$ equipped with a bi-invariant metric.
\end{lemma}
\begin{proof}
Let us denote by $\mathsf{K}$ the complete Riemannian manifold corresponding to distribution $\mathcal{K}$. By construction, the Bismut connection $\nabla$ of $(M,J,g)$ restricts to a flat Riemannian connection with skew-symmetric torsion on $\mathsf{K}$, that we will still denote by $\nabla$. Let now $\mathcal{K} = \mathcal{K}_1 + {\dots} + \mathcal{K}_p$ be the decomposition of the Lie algebra bundle $\mathcal{K}$ corresponding to the splitting of $\mathfrak{k}$ into simple ideals $\mathfrak{k}_1, {\dots}, \mathfrak{k}_p$. Being $\nabla$-flat, each $\mathcal{K}_i$ is also Riemannian-irreducible, and so $\mathsf{K}$ splits as a product of complete Riemannian manifolds
$$
\mathsf{K} = \mathsf{K}_1 \times {\dots} \times \mathsf{K}_p \,\, .
$$
Therefore, we can apply \cite[Theorem 2.2]{MR2651537}. In particular, by \eqref{eq:BKL}, we fall into the case $\sigma_T=0$ in the above-mentioned theorem. Consequently, we can conclude that each $\mathsf{K}_i$ is a compact, simple Lie group and the induced metric is bi-invariant.
\end{proof}

We are finally ready to gather all the results needed to obtain Theorem \ref{thm:MAIN-A}.

\begin{proof}[Proof of Theorem \ref{thm:MAIN-A}]
Let $(M^{2n}, J, g)$ be a complete, simply-connected Hermitian manifold. If it decomposes as a product of Hermitian irreducible factors, each of them is either K\"ahler or a Riemannian product as in \eqref{eq:usualproduct} endowed with a standard complex structure, then $(M^{2n}, J, g)$ is pluriclosed with parallel Bismut torsion by Theorem \ref{thm:standard-BKL}. On the other hand, if $(M^{2n}, J, g)$ is pluriclosed with parallel Bismut torsion, then each Hermitian irreducible non-K\"ahler factor of $M$ is of the form \eqref{eq:usualproduct} by \eqref{eq:finaldec}, Lemma \ref{lem:sasaki} and Lemma \ref{lem:Sam}. Moreover, the complex structure $J$ on each non-K\"ahler factor is standard (see Definition \ref{def:standard}). Indeed, it is compatible with the transverse complex distribution on each Sasaki factor and it projects onto a $\mathsf{K}$-invariant complex structure on the full flag manifold $\mathsf{K}/\mathsf{T}$ (see \eqref{eq:quasifinaldec}, \eqref{eq:finaldec} and Lemma \ref{lem:sasaki}). Finally, it preserves the distribution $\mathcal{Z} +\mathcal{T}$ and it acts on the global generators $(H_1,{\dots},H_{2m})$ as in \eqref{eq:JH}, since all the vector fields $H_i$, $JH_i$ are $\nabla$-parallel.
\end{proof}

\medskip
\section{Classification in low dimensions}
\label{sect:low-dim} \setcounter{equation} 0

In this section, we show how Theorem \ref{thm:MAIN-A} can be used to recover all the previous classification results in low dimension. We start with the $2$-dimensional case, where it is known that the BKL condition is equivalent to being Vaisman. Before stating the result, let us recall the definition of a Vaisman metric. Given a Hermitian manifold $(M^{2n},J,g)$, its {\it Lee form} is the $1$-form $\phi$ uniquely determined by the equation
$$
{\rm d}\omega^{n-1} = \omega^{n-1} \wedge \phi \,\, ,
$$
where $\omega = g(J\cdot\,,\cdot\,)$. Then, the Hermitian metric $g$ is said to be {\it Vaisman} if $D\phi = 0$ and satisfies
$$
{\rm d}\omega = \frac{1}{n-1}\, \omega \wedge \phi \,\, .
$$

\begin{theorem}[see Theorem 2 in \cite{MR4554474}]
Let $(M^4,J)$ be a complex surface and $g$ a complete Hermitian metric on it. If $(M^4, J, g)$ is BKL, then the metric $g$ is Vaisman.
\end{theorem}

\begin{proof}
Denote by $\phi$ the Lee form of $(M,J,g)$. By Theorem \ref{thm:MAIN-A}, the Hermitian universal cover of $(M,J,g)$ is holomorphically isometric to the product $\mathbb{R} \times S$, where $S$ is a Sasaki $3$-manifold, endowed with a standard complex structure according to Definition \ref{def:standard}. More precisely, if we denote by $H$ the Reeb vector field of $S$, then $JH$ is tangent to the factor $\mathbb{R}$ and we can consider local orthonormal generators $(E,JE)$ of the orthogonal complement $H^{\perp} \cap TS$. Following \eqref{eq:diffcofSas} and \eqref{eq:diffcofRell}, the differential of the corresponding covectors $E^*,JE^*, H^*, JH^*$ are given by
$$\begin{aligned}
&{\rm d}H^* = 2c E^* \wedge JE^* \,\, , \\
&{\rm d}JH^* = 0 \,\, , \\
&{\rm d}E^* = cJE^* \wedge H^* -\vartheta \wedge JE^* \,\, ,\\
&{\rm d}JE^* = -cE^* \wedge H^* +\vartheta \wedge E^* \,\, ,
\end{aligned}$$
where $\vartheta$ is a local $1$-form on $S$ and $c>0$. Therefore, the corresponding fundamental $2$-form $\omega$ can be locally written as
$$
\omega = E^*\wedge JE^* + H^*\wedge JH^*
$$
and so
$$
{\rm d}\omega = 2c E^*\wedge JE^* \wedge JH^* = 2c\, \omega \wedge JH^* \,\, ,
$$
which means that $2c\,JH^*$ is the pullback of the Lee form $\phi$ on the universal cover. Since $JH^*$ is $D$-parallel, this implies that $(M,J,g)$ is a Vaisman surface.
\end{proof}

We now move to the classification of BKL manifolds in complex dimensions $3$, $4$, and $5$, obtained in \cite{MR4577328} and \cite{ZZ4-5}. Before doing this, we recall that, according to their terminology, a BKL manifold is {\it full} if and only if its Hermitian universal cover admits no K\"ahler de Rham factors. It is then reasonable to classify full BKL manifolds, since any BKL manifold splits as the product of a K\"ahler manifold and a full BKL manifold. By Theorem \ref{thm:MAIN-A} and the classification of compact, simply-connected simple Lie groups (see {\it e.g.\ }\cite[Chapter V]{MR0781344}), we recover the following

\begin{theorem}[see Theorem 7 in \cite{MR4577328}, Theorem 4 and Theorem 5 in \cite{ZZ4-5}]
Let $(M^{2n},J,g)$ be a complete, full, BKL manifold of complex dimension $n$ and $(\widetilde{M},J,g)$ its Hermitian universal cover.
\begin{itemize}
\item[i)] If $n=3$, then $(\widetilde{M},J,g)$ is holomorphically isometric to the product $S_1 \times S_2$ of two Sasaki $3$-dimensional manifolds, equipped with a standard complex structure.
\item[ii)] If $n=4$, then $(\widetilde{M},J,g)$ is holomorphically isometric to $\mathsf{SU}(3)$ with a Bismut flat Hermitian structure or the product $\mathbb{R}^2 \times S_1 \times S_2$, equipped with a standard complex structure.
\item[iii)] If $n=5$, then $(\widetilde{M},J,g)$ is holomorphically isometric to $\mathsf{Spin}(5)$ or $\mathsf{SU}(3) \times \mathbb{R}^2$, equipped with a Bismut flat Hermitian structure, or the product $\mathbb{R} \times S_1 \times S_2 \times S_3$, equipped with a standard complex structure.
\end{itemize}
\end{theorem}

With the same approach, one can also write down the possibilities in higher dimensions. For example, we list here the cases in complex dimension $6$.

\begin{theorem}
Let $(M^{12},J,g)$ be a complete, full, BKL manifold of complex dimension $6$. Then, its Hermitian universal cover $(\widetilde{M},J,g)$ is holomorphically isometric to one of the following:
\begin{itemize}
\item[i)] $\mathbb{R}^2 \times \mathsf{Spin}(5)$, equipped with a Bismut flat Hermitian structure;
\item[ii)] $\mathbb{R} \times S \times \mathsf{SU}(3)$, where $S$ is a Sasaki $3$-dimensional manifold,  equipped with a standard complex structure;
\item[iii)] the product $S_1 \times S_2 \times S_3 \times S_4$ of four Sasaki $3$-dimensional manifolds, equipped with a standard complex structure;
\item[iv)] the product $\mathbb{R}^3 \times S_1 \times S_2 \times S_3$ of three Sasaki $3$-dimensional manifolds and $\mathbb{R}^3$, equipped with a standard complex structure.
\end{itemize}
\end{theorem} 

\begin{remark}
For the sake of clarity, we stress the following fact. Differently from Section \ref{sect:classification}, in order to simplify the statements of the theorems in this section, we address the homogeneous manifold $\mathsf{SU}(2)$ as a Sasaki $3$-dimensional manifold.
\end{remark}

Notice that the first example of complete, simply-connected, full, non-K\"ahler and non Bismut flat, BKL manifold which is not just the product of Sasaki $3$-manifolds and Euclidean factors, namely $\mathsf{SU}(3) \times S \times \mathbb{R}$, appears in complex dimension $6$. This example shows that, in \cite{MR4577328, MR4554474, ZZ4-5}, the authors find only BKL manifolds that were either Bismut flat or products $\mathbb{R}^{\ell} \times \Pi_i S_i$ because of dimensional reasons. On the contrary, by Proposition \ref{prop:trick}, the presence of non Bismut flat Sasaki components of dimension $d > 3$ on these manifolds is obstructed by geometric reasons (see also \cite[Proposition 5.6]{MR4733370}). \smallskip

For the convenience of the reader, we point out that, in \cite{MR4577328, MR4554474, ZZ4-5}, the authors conduct their analysis distinguishing different cases by means of the rank $r_B$ of a suitable bilinear form (see \cite[Definition 2]{ZZ4-5}). In our notation, $r_B$ coincides with the quantity
\begin{equation} \label{eq:rB}
r_B = \frac{1}{2}\big(\dim_\mathbb{R} M -\dim(\mathfrak{t} +\mathfrak{z})\big) = \frac{1}{2}(2n -(\ell +s +r)) = n-m \,\, ,
\end{equation}
where $n = \dim_{\mathbb{C}} M$ and the numbers $\ell, s, r, m \in \mathbb{N}$ are as in the beginning of Section \ref{sect:if-part} (see \eqref{eq:standardBKL}). Notice that, by Remark \ref{rem: center dimension} (where, by definition, $q=s+r$), the condition of being full implies that
$$
\frac{n}{2} \leq r_B \leq n-1 \,\, ,
$$
see also \cite[Theorem 1]{ZZ4-5}. For $n=2$ and $n=3$, the rank $r_B$ gives no additional information about the geometry of BKL manifolds. For $n \geq 4$, the case $r_B = \frac{n}{2}$ is reached when $\dim(\mathfrak{t} +\mathfrak{z}) = n$, while the case $r_B = n -1$ corresponds to $\dim(\mathfrak{t} +\mathfrak{z}) = 2$. Then, Theorem \ref{thm:MAIN-A} reduces the problem of characterizing these extremal cases to a Lie theoretical classification problem, namely, finding real compact Lie algebras of dimension $2n$ with rank $n$ and $2$, respectively. Hence, we recover the following result, originally obtained in \cite{ZZ4-5}.

\begin{theorem}[Theorem 2 and Theorem 4 in \cite{ZZ4-5}]
Let $(M^{2n},J,g)$ be a complete, full, BKL manifold of complex dimension $n \geq 4$ and $r_B$ as in \eqref{eq:rB}.
\begin{itemize}
\item[i)] If $r_B = \frac{n}{2}$, the Hermitian universal cover $(\widetilde{M},J,g)$ is holomorphically isometric to the product $S_1 \times {\dots} \times S_{r_B}$ of $r_B$ Sasaki $3$-dimensional manifolds, equipped with a standard complex structure.
\item[ii)] If $r_B = n-1$, then $(M,J,g)$ is Bismut flat.
\end{itemize}
\end{theorem}

\medskip
\section{Pluriclosed CYT manifolds with parallel Bismut torsion}
\label{sect:BKL+CYT} \setcounter{equation} 0

From the classification obtained in Theorem \ref{thm:MAIN-A}, it is evident that BKL manifolds without K\"ahler de Rham factors are close to being Bismut flat manifolds. More precisely, the Bismut curvature tensor vanishes everywhere except, possibly, on the Sasaki factors $S_i$ (see Proposition \ref{prop:flatwhite} and Proposition \ref{prop:trick}). Moreover, the Reeb vector field $H_{\ell +i}$ of $S_i$ is $\nabla$-parallel for any $1 \leq i \leq s$, therefore, the Bismut curvature $R$ may be nonzero only on the $2$-dimensional transverse Sasaki distribution $H_{\ell +i}^{\perp} \cap TS_i$ (see Lemma \ref{lem:almost-flat} ). By dimensional reasons, it follows that the Bismut Ricci form \eqref{eq:BRicci} carries all the information about the curvature. \smallskip

Thanks to this basic idea, we recover \cite[Theorem 3, Remark 1]{ZZ4-5} and \cite[Theorem 1.1]{brienza2024cyt}, concerning the sign of the Bismut Ricci form. For convenience, we recall that $(M,J,g)$ is said to be {\it Calabi--Yau with torsion} ({\it CYT} for short) if its Bismut Ricci form $\rho$ vanishes.

\begin{theorem}[see Theorem 3 and Remark 1 in \cite{ZZ4-5}, Theorem 1.1 in \cite{brienza2024cyt}] \label{thm:univcoverBKLCYT}
Let $(M, J, g)$ be a complete BKL manifold and $(\widetilde{M}, J, g)$ its Hermitian universal cover. Denote by $\rho$ the Bismut Ricci form of $(M, J, g)$. \begin{itemize}
\item[i)] If $(M, J, g)$ is CYT, then $(\widetilde{M}, J, g)$ is holomorphically isometric to the product of a K\"ahler Ricci-flat manifold and a Bismut flat manifold.
\item[ii)] If $(\widetilde{M}, J, g)$ contains no K\"ahler de Rham factors, then the Bismut bisectional curvature of $(M, J, g)$ is non-negative (resp.\ non-positive) if and only if $\rho$ is non-negative (resp.\ non-positive).
\end{itemize}
\end{theorem}

\begin{proof}
From Theorem \ref{thm:MAIN-A} we know that $\widetilde{M}$ decomposes as a product of Hermitian irreducible factors, each of them being either K\"ahler or a Riemannian product
$$
\mathbb{R}^{\ell} \times \prod_{i=1}^{s} S_{i} \times \mathsf{K}
$$
of a Euclidean component $\mathbb{R}^{\ell}$, $s$ Sasaki $3$-dimensional manifolds $S_i$, and a compact semisimple Lie group $\mathsf{K}$, of rank $r$, with a bi-invariant metric, endowed with the standard complex structure as in Definition \ref{def:standard}, with $\ell \leq s+r$. Here, following Section \ref{sect:classification}, we are assuming that each Sasaki factor $S_i$ is not the homogeneous manifold $\mathsf{SU}(2)$. It is then sufficient to focus on each non-K\"ahler irreducible factor, since the Bismut Ricci form splits accordingly and coincides with the Riemannian Ricci form on the K\"ahler factors.

Fix then a non-K\"ahler irreducible factor. By Lemma \ref{lem:almost-flat} , its Bismut curvature tensor $R$ reduces to $s$ scalar functions and
$$
\rho(E_i,JE_i)= R(E_i,JE_i,JE_i,E_i) \,\, ,
$$
where $(E_i,JE_i)$ are local unitary frames for $H_{\ell +i}^{\perp} \cap TS_i$, $1 \leq i \leq s$. Therefore, $\rho = 0$ if and only if $R=0$, which corresponds to the case $s=0$. Moreover, if $s >0$, then the sign of the Bismut bisectional curvature is completely determined by the the sign of $\rho$, and this concludes the proof. 
\end{proof}

\begin{rem} \label{rem:Samelson}
During this proof, we observed the following well-known fact: a simply-connected, Bismut flat manifold is holomorphically isometric to an even dimensional Lie group $\mathbb{R}^{\ell} \times \mathsf{K}$, where $\mathsf{K}$ is compact and semisimple, endowed with a bi-invariant metric and a left-invariant complex structure. Following \cite{MR4127891}, we refer to them as {\it Samelson spaces}.
\end{rem}

\begin{rem}
From the above proof it is also clear that, for dimensional reasons, on a BKL surface, that is a Vaisman surface in virtue of \cite[Theorem 2]{MR4554474}, the Bismut bisectional curvature reduces to a single scalar function, that coincides with the {\it Bismut scalar curvature}.
\end{rem}

Finally, we prove a structure result for compact CYT and BKL manifolds. Before doing that, we provide a proof for the following lemma, which is well-known to the experts.

\begin{lemma} \label{lem:IsodeRham}
Let $(\widetilde{M},g)$ be a simply-connected, complete Riemannian manifold and let
\begin{equation} \label{eq:deRhamdec}
(\widetilde{M},g) \simeq \mathbb{R}^{\ell} \times (\widetilde{M}_1,g_1) \times {\dots} \times (\widetilde{M}_k,g_k)
\end{equation}
its de Rham decomposition. Then, the full isometry group ${\rm Iso}(\widetilde{M},g)$ splits as
\begin{equation} \label{eq:IsodeRham}
{\rm Iso}(\widetilde{M},g) \simeq {\rm Iso}(\mathbb{R}^{\ell}) \times \Big(\mathfrak{S} \ltimes \big({\rm Iso}(\widetilde{M}_1,g_1) \times {\dots} \times {\rm Iso}(\widetilde{M}_k,g_k)\big)\Big) \,\, ,
\end{equation}
where $\mathfrak{S}$ denotes the finite group that permutes the isometric factors of \eqref{eq:deRhamdec}.
\end{lemma}

\begin{proof}
By \cite[Sect VI.3, Lemma 2]{MR1393940}, there is a injective group homomorphism
$$
{\rm Iso}(\mathbb{R}^{\ell}) \times {\rm Iso}(\widetilde{M}_1,g_1) \times {\dots} \times {\rm Iso}(\widetilde{M}_k,g_k) \to {\rm Iso}(\widetilde{M},g) \,\, .
$$
Moreover, by \cite[Sect VI.3, Theorem 3.5]{MR1393940}, we know that 
$$
\dim {\rm Iso}(\mathbb{R}^{\ell}) + \sum_{1 \leq i \leq k} \dim {\rm Iso}(\widetilde{M}_i,g_i) = \dim{\rm Iso}(\widetilde{M},g) \,\, .
$$
It remains to show that any isometry of $(\widetilde{M},g)$ must preserve or interchange the irreducible factors in \eqref{eq:deRhamdec}. Let $\varphi \in {\rm Iso}(\widetilde{M},g)$ be an isometry, $p \in \widetilde{M}$ a point and consider the holonomy decomposition
$$
T\widetilde{M} = V^0 + V^1 + {\dots} + V^{k}
$$
of the tangent bundle $T\widetilde{M}$. By \cite[Sect VI.3, Lemma 2]{MR1393940}, the differential ${\rm d}\varphi|_p$ preserves the holonomy decomposition up to order, {\it i.e.}, ${\rm d}\varphi|_p (V^0_{p}) = V^0_{\varphi(p)}$ and there exists a permutation $\sigma$ of $k$ elements such that ${\rm d}\varphi|_p (V^i_{p}) = V^{\sigma(i)}_{\varphi(p)}$ for any $1 \leq i \leq k$. Write $p \simeq (p_0,p_1,{\dots},p_k)$ and $\varphi(p)\simeq (q_0,q_1,{\dots},q_k)$ according to \eqref{eq:deRhamdec}. Then, since any factor in \eqref{eq:deRhamdec} is connected, complete and totally geodesic, and $\varphi$ is an isometry, it follows that
$$
\varphi\big(\mathbb{R}^{\ell} \times \{p_1\} \times {\dots} \times \{p_k\}\big) = \mathbb{R}^{\ell} \times \{q_1\} \times {\dots} \times \{q_k\}
$$
and
$$
\varphi\big(\{p_0\} \times \{p_1\} \times {\dots} \times \widetilde{M}_i \times {\dots} \times \{p_k\}\big) = \{q_0\} \times \{q_1\} \times {\dots} \times \widetilde{M}_{\sigma(i)} \times {\dots} \times \{q_k\}
$$
for any $1 \leq i \leq k$. Clearly $(\widetilde{M}_i,g_i)$ is isometric to $(\widetilde{M}_{\sigma(i)},g_{\sigma(i)})$ and, since $\varphi$ is smooth, the permutation $\sigma$ does not depend on the choice of the point $p$. This concludes the proof.
\end{proof}

Note that the analogous of Lemma \ref{lem:IsodeRham} also holds for compact, non-simply-connected manifold (see \cite[Corollary 1]{MR1473665}). By using Lemma \ref{lem:IsodeRham}, we get the following

\begin{theorem} \label{thm:BismutEinstein}
Let $(M,J,g)$ be a compact, CYT and BKL manifold. Then it splits, up to a finite cover, as a Hermitian product $M = \mathcal{Y} \times N$, where $\mathcal{Y}$ is a compact, simply-connected, K\"ahler Ricci flat manifold and $N$ is a compact, Bismut flat manifold.
\end{theorem}

\begin{proof}
By Theorem \ref{thm:univcoverBKLCYT} and Remark \ref{rem:Samelson}, the Hermitian universal cover $(\widetilde{M},J,g)$ of $(M,J,g)$ splits isometrically as
\begin{equation} \label{eq:BKLCYTdecuniv}
(\widetilde{M},J,g) = \mathcal{Y} \times \mathbb{R}^{\ell} \times \mathsf{K} \,\, ,
\end{equation}
where $\mathcal{Y}$ is a simply-connected, complete K\"ahler Ricci flat manifold and $\mathbb{R}^{\ell} \times \mathsf{K}$ is a Samelson space. Therefore, $(M,g)$ has non-negative Riemannian Ricci curvature. We can assume that $\mathcal{Y}$ does not have any Riemannian flat factors, as they can be gathered with the factor $\mathbb{R}^{\ell}$ appearing in \eqref{eq:BKLCYTdecuniv}. Consequently, by the Cheeger--Gromoll Theorem \cite[Theorem 3]{MR303460}, it follows that $\mathcal{Y}$ is compact. Since $\mathcal{Y}$ is compact, simply-connected and its Riemannian Ricci tensor vanishes, the Bochner formula implies that the group ${\rm Iso}(\mathcal{Y})$ is finite.

Observe that, by \cite[Sect IX.8, Theorem 8.1]{MR1393941}, each de Rham factor of $\mathcal{Y}$ inherits a K\"ahler structure. Therefore, no irreducible de Rham factor of $\mathcal{Y}$ can be isometric to an irreducible de Rham factor of $\mathsf{K}$. Consequently, by \eqref{eq:IsodeRham}, we get
\begin{equation} \label{eq:IsodeRham2}
{\rm Iso}(\widetilde{M},g) = {\rm Iso}(\mathcal{Y}) \times {\rm Iso}(\mathbb{R}^{\ell}) \times {\rm Iso}(\mathsf{K}) \,\, .
\end{equation}
Let now $\pi_1(M)$ be the fundamental group of $M$, that acts on $\widetilde{M}$ via deck transformations of the cover by holomorphic isometries. By \eqref{eq:IsodeRham2}, we get an injective group homomorphism
$$
\tau = \big(\tau_{\mathcal{Y}},\tau_{\mathbb{R}^{\ell}},\tau_{\mathsf{K}}\big) : \pi_1(M) \to {\rm Iso}(\mathcal{Y}) \times {\rm Iso}(\mathbb{R}^{\ell}) \times {\rm Iso}(\mathsf{K}) \,\, .
$$
Since ${\rm Iso}(\mathcal{Y})$ is finite, ${\rm ker}(\tau_{\mathcal{Y}})$ is a normal subgroup of $\pi_1(M)$ with finite index. Therefore, up to passing to a finite cover, we may replace $\pi_1(M)$ by ${\rm ker}(\tau_{\mathcal{Y}})$ (see \cite[Proposition 1.36]{MR1867354}). This implies that $M$ is finitely covered by the product $\mathcal{Y} \times N$, where $N$ is compact and Bismut flat.
\end{proof}

We are finally ready to gather all the results needed to obtain Theorem \ref{thm:MAIN-B}.

\begin{proof}[Proof of Theorem \ref{thm:MAIN-B}]
By Theorem \ref{thm:BismutEinstein}, it follows that any CYT and BKL manifold $(M,J,g)$ splits, up to a finite cover, as a product of a compact, simply-connected, K\"ahler Ricci flat manifold $\mathcal{Y}$ and compact, Bismut flat manifold $N$. Moreover, by \cite[Theorem 1]{MR4127891}, it follows that $N$ is finitely covered by a local Samelson space $M'$. Finally, it is straightforward to check that any product $\mathcal{Y} \times M'$ as above is CYT and BKL.
\end{proof}

\end{document}